\documentclass[12pt]{amsart}
\newif\ifdraft
\draftfalse

\ifdraft\usepackage[notref,notcite]{showkeys}\fi

\usepackage{indentfirst}
\usepackage{bm,hyperref, mathrsfs, graphics,amssymb, amsmath, amsfonts, subeqnarray,setspace,graphicx,amsthm,epstopdf, mathtools, subfigure,color,mathrsfs}

-\marginparwidth \oddsidemargin -9mm \evensidemargin \oddsidemargin
\advance\hoffset by 5mm

\def\b{\boldsymbol}

\def\e{\epsilon}

\newcommand{\tcb}[1]{\textcolor{blue}{#1}}

\DeclareMathOperator{\tr}{tr}

\newtheorem{thm}{Theorem}
\newtheorem{lmm}{Lemma}
\newtheorem{pro}{Proposition}
\newtheorem{rmk}{Remark}
\newtheorem{cor}{Corollary}
\newtheorem{definition}{Definition}

\makeatletter
\@mparswitchfalse
\makeatother

\newcommand\blfootnote[1]{%
  \begingroup
  \renewcommand\thefootnote{}\footnote{#1}%
  \addtocounter{footnote}{-1}%
  \endgroup
}

\begin{document}

\title{$p$-Euler equations and $p$-Navier-Stokes equations}

\author{Lei Li}
\address{\hskip-\parindent
Lei Li\\
Department of Mathematics\\
Duke University\\
Durham, NC 27708, USA
}
\email{leili@math.duke.edu}

\author{Jian-Guo Liu}
\address{\hskip-\parindent
Jian-Guo Liu\\
Departments of Physics and Mathematics\\
Duke University\\
Durham, NC 27708, USA
}
\email{jliu@phy.duke.edu}

\begin{abstract}
We propose in this work new systems of equations which we call $p$-Euler equations and $p$-Navier-Stokes equations. $p$-Euler equations are derived as the Euler-Lagrange equations for the action represented by the Benamou-Brenier characterization of Wasserstein-$p$ distances, with incompressibility constraint. $p$-Euler equations have similar structures with the usual Euler equations but the `momentum' is the signed ($p-1$)-th power of the velocity. In the 2D case, the $p$-Euler equations have streamfunction-vorticity formulation, where the vorticity is given by the $p$-Laplacian of the streamfunction. By adding diffusion presented by $\gamma$-Laplacian of the velocity, we obtain what we call $p$-Navier-Stokes equations. If $\gamma=p$, the {\it a priori} energy estimates for the velocity and momentum have dual symmetries. Using these energy estimates and a time-shift estimate, we show the global existence of weak solutions for the $p$-Navier-Stokes equations in $\mathbb{R}^d$ for $\gamma=p$ and $p\ge d\ge 2$ through a compactness criterion.
\end{abstract}

\maketitle


\section{Introduction}

The Wasserstein distances \cite{gangbomccann95,bb99,villani08,santambrogio15} for probability measures in a domain $O\subset\mathbb{R}^d$ are closely related to optimal transport and are useful for image processing \cite{ss2013}, machine learning \cite{acb17} and fluid mechanics \cite{bb00}. If $O$ is convex and bounded, the Wasserstein-$p$ ($p>1$) distance between two probability measures $\mu,\nu$ in $O$ can be reformulated as the following optimization problem \cite[Sec. 5.4]{santambrogio15}:
\begin{gather*}
W_p^p(\mu, \nu)
=\min_{\rho,m}\left\{\int_0^1 p\mathscr{B}_p(\rho, m)dt: \partial_t\rho+\nabla\cdot m=0, \rho|_{t=0}=\mu, \rho|_{t=1}=\nu \right\},
\end{gather*}
where $\rho$ is a nonnegative measure and $m$ is a vector measure, both of which are time-dependent. $\mathscr{B}_p$ is the Benamou-Brenier functional, and see Equation \eqref{eq:bbfunctional} for the expression  in the case $m\ll \rho$ (i.e. $m$ is absolutely continuous with respect to $\rho$). This Benamou-Brenier characterization of Wasserstein distances provides a least action principle framework for us to study Wasserstein geodesics. 

In applications like image processing, one usually wants to find the geodesics between two shapes using some suitable action \cite{ss2013}. In \cite{lps16}, Liu et al. considered two shapes (open connected sets) $\Omega_0$ and $\Omega_1$ in $O$ with equal volume $|\Omega_0|=|\Omega_1|$. Assigning the two shapes with uniform probability measures
\[
\mu=\frac{1}{|\Omega_0|}\chi(\Omega_0),~\nu=\frac{1}{|\Omega_1|}\chi(\Omega_1),
\]
where $\chi(E)$ for a set $E$ means the characteristic function, the Wasserstein-$p$ distance between these two shapes is defined as the Wasserstein-$p$ distance between $\mu$ and $\nu$. The authors of \cite{lps16} considered the geodesics between $\Omega_0$ and $\Omega_1$ with the action represented by the Benamou-Brenier characterization of  Wasserstein-$2$ distance under incompressibility constraint. In other words, they studied the action
\[
\mathcal{A}=\frac{1}{2}\int_0^1\int_{O}\rho |v|^2 \,dx dt,
\]
with the constraint
\[
|\Omega_t|=|\Omega_0|,~\rho(\cdot, t)=\frac{1}{|\Omega_t|}\chi(\Omega_t),~\forall t\in[0,1].
\]
Note that the action $\mathcal{A}$ here is different from the one used in \cite{lps16} by a multiplicative constant, to be consistent with our convention in this paper. They found that the Euler-Lagrange equations for the geodesics under this constraint are the incompressible, irrotational Euler equations with free boundary. They proved that the distance between two shapes under this notion with incompressibility constraint is equal to the Wasserstein-$2$ distance.

The work in \cite{lps16} is related to Arnold's least action principle \cite{arnold66}, where Arnold discovered that the Euler equations of inviscid fluid flow can be viewed as the geodesic path in the group of volume-preserving diffeomorphisms in a fixed domain. One of the differences is that the equations in \cite{lps16} are irrotational with free boundary. The free boundary problems for incompressible Euler equations are waterwave equations which have attracted a lot of attention \cite{stokes1847,bhl93,wu97,wu99,shatahzeng08}. In \cite{wu97,wu99}, Wu proved the wellposedness of waterwave problems in Sobolev spaces with general data for irrotational, no surface tension cases. In \cite{shatahzeng08}, Shatah and Zeng solved the zero tension limit problem with the observation that the Lagrange multiplier part of the pressure can be interpreted as the second fundamental form of the manifold of the flow map.

In this paper, following \cite{lps16}, we derive the Euler-Lagrange equation for the action represented by the Benamou-Brenier characterization of Wasserstein-$p$ distance ($p>1$) between two shapes with the incompressibility constraint. The resulted equations (Equations \eqref{eq:pEuler})  have similar structures with Euler equations for $x\in \Omega_t$:
\begin{gather*}
\left\{
\begin{split}
&\partial_t v_p+v\cdot\nabla v_p=-\nabla \pi ,\\
& v_p=|v|^{p-2}v,\\
&\nabla\cdot v=0,\\
\end{split}\right.
\end{gather*} 
and thus we call them $p$-Euler equations. Here $v$ is the Eulerian velocity field for the particles along the geodesics while $v_p$, which we call the momentum, is the signed power of velocity $v$. $\pi$ is a scalar field which plays the role of pressure as in the usual Euler equations. The $p$-Euler equations for the geodesics are irrotational in the sense that $\nabla\times v_p=0$.

If we define the Lagrangian and Hamiltonian respectively as (Equations \eqref{eq:lag} and \eqref{eq:hal}):
\[
L(v)=\int_{\Omega}\frac{1}{p}|v|^p dx,~~
H(v_p)=\int_{\Omega}\frac{1}{q}|v_p|^qdx,
\]
where $q$ is the conjugate index of $p$ ($1/p+1/q=1$), then $pL=qH$ and we have then the dual symmetry (Equation \eqref{eq:dualsym})
\[
v_p=\frac{\delta L}{\delta v},~v=\frac{\delta H}{\delta v_p}.
\]
Further, the $p$-vorticity
\[
\omega_p:=\nabla\times v_p,
\] 
evolves like a material transported by the velocity field:
\[
\partial_t\omega_p+v\cdot\nabla\omega_p=(\omega_p\cdot\nabla)v.
\]
If the flow is irrotational in the sense that $\omega_p=0$, then $v_p=\nabla\phi$ and we have the following Bernoulli equations (Proposition \ref{pro:formulationofpEuler})
\begin{gather*}
\left\{
\begin{split}
&\partial_t\phi+\frac{1}{q}|\nabla \phi|^q+\pi=c(t),\\
&-\Delta_q\phi=-\nabla\cdot(|\nabla\phi|^{q-2}\nabla\phi)=0,
\end{split}
\right.
\end{gather*}
for some arbitrary function $c(t)$. The $q$-Laplacian of the potential is zero.

For the free boundary problems, we use Noether's theorem to find the conservation of Hamiltonian $H$ and total momentum, angular momentum, an analogy of helicity and  Kelvin's circulation (see Section \ref{sub:symm}):
\[
\frac{d}{dt}H=0,~\frac{d}{dt}\int_{\Omega_t}v_p\, dx=0,
~\frac{d}{dt}\int_{\Omega_t}x\times v_p\, dx=0,
~\frac{d}{dt}\int_{\Omega_t}\omega_p\cdot v_p\, dx=0,
~\frac{d}{dt}\int_{\tilde{\gamma}}v_p\cdot \tau\, ds=0, 
\]
where in the last integral $\tilde{\gamma}$ is a material closed curve while $\tau=\frac{d}{ds}\tilde{\gamma}$ is the unit tangent vector with $s$ being the arc length parameter. 
For a fixed boundary, we do not have conservation of momentum and conservation of angular momentum. However, the conservation of Hamiltonian, helicity and circulation are still valid.

When we consider the 2D $p$-Euler equations, we have a streamfunction-vorticity formulation for $x\in\Omega$:
\begin{gather*}
\left\{
\begin{split}
& \partial_t\omega_p+v\cdot\nabla\omega_p=0,\\
&-\Delta_p\psi=-\nabla\cdot(|\nabla\psi|^{p-2}\nabla\psi)=\omega_p,\\
&v=\nabla^{\perp}\psi. \\
\end{split}
\right.
\end{gather*}
The vorticity becomes $p$-Laplacian of the stream function. This system of equations share similarities with the surface quasi-geostrophic equations, where the $p$-Laplacian is replaced by a fractional Laplacian \cite{cmt94}
\[
-(-\Delta)^{1/2}\psi=\omega.
\]
The well-posedness of critical 2D dissipative quasi-geostrophic equation was shown in \cite{knv07,cv10}. Another related system is the semi-geostrophic equations where the relation between $\psi$ and $\omega$ is given by a Monge-Ampe\'re equation \cite{loeper06}:
\[
\mathrm{det}\,(I+\nabla^2\psi)=\omega.
\]

Adding viscous term given by $\gamma$-Laplacian of velocity,  which is reminiscent of the shear-thinning or thickening velocity-dependent viscosity for non-Newtonian fluids \cite{los69,leibeson83,breit2017}, we obtain what we call $p$-Navier-Stokes ($p$-NS) equations (Equation \eqref{eq:pNS}) for $x\in\Omega\subset\mathbb{R}^d$:
\begin{gather*}
\left\{
\begin{split}
&\displaystyle\partial_t v_p+v\cdot\nabla v_p=-\nabla \pi+\nu \Delta_{\gamma} v,\\
&\displaystyle v_p=|v|^{p-2}v,\\
&\displaystyle\nabla\cdot v=0.\\
\end{split}
\right.
\end{gather*}
Here, $\nu>0$ and $\gamma>1$ are constants. Note that the name `$p$-Navier-Stokes equations' is reminiscent of the models for non-Newtonian fluids studied by Breit in \cite{breit2015,breit2017} based on a power law model for the viscosity term (see Remark \ref{rmk:nonnewtonian} for more details) \footnote{\tcb{Actually, the name `$p$-Navier-Stokes equations' has been used by Breit in his talk slides for some of these models: \url{http://www.macs.hw.ac.uk/~b13/Habil.pdf.}} }. 
The $p$-NS equations have scaling invariance (see Section \ref{subsec:scale}) and therefore may admit self similar solutions. By studying self-similar solutions with $L^2(\mathbb{R}^3)$ initial data, Jia and Sverak in \cite{jiasverak15} studied the nonuniqueness of Leray-Hopf weak solutions of the 3D Navier-Stokes equations.

The parameter $\gamma$ measures the strength of diffusion.  $\gamma\in(1,2)$ corresponds to fast diffusion while $\gamma\in (2,\infty)$ corresponds to slow diffusion. Special cases include $\gamma=2, \gamma=p$. For $\gamma=2$ and in 2D case, the $p$-NS equations can be recast in a streamfunction-vorticity formulation (see Section \ref{sec:pns}). In the $\gamma=p$ cases, the $p$-NS equations have dual symmetry (Equation \eqref{eq:dualsym})  for the momentum and velocity. In particular, the dual symmetry gives the following energy-dissipation relation
 \[
 \frac{dH}{dt}=-\nu \int_{\Omega} |\nabla v|^p dx,
 \] 
 and from which we observe the important {\it a priori} energy estimates
 for $v_0\in L^p(\Omega)$:
\begin{gather*}
v\in L^{\infty}(0,T; L^p(\Omega))\cap L^p(0,T; W^{1,p}(\Omega)), ~ v_p\in L^{\infty}(0,T; L^q(\Omega))\cap L^q(0,T; W^{1,q}(\Omega)),
\end{gather*} 
where $q$ is the conjugate index of $p$.  For $\Omega=\mathbb{R}^d$, we have a time-shift estimate  (Lemma \ref{lmm:timeestimate}) for the mollified sequence
\[
\|\tau_hv^{\e}-v^\e\|_{L^p(0,T-h; L^p(\mathbb{R}^d))}\to 0,~\text{uniformly as } h\to 0+,
\] 
where the time-shift operator is given by $\tau_h f(t)=f(t+h)$. Using these estimates, we conclude the compactness of $v^{\e}$ in $L^p(0,T; L^p(\mathbb{R}^d))$ by a compactness theorem from \cite{cl12} and show the global existence of weak solutions in $\mathbb{R}^d$ for $p\ge d\ge 2$ with $\gamma=p$ in Section \ref{sec:weaksol}. 

Diffusion with $p$-Laplacian diffusion has been studied by many authors.  The authors of \cite{bernis88} and \cite{mm12} studied the weak solutions of doubly degenerate diffusion equations  and in  \cite{aac10}, Agueh et al. investigated the large time asymptotics of doubly nonlinear diffusion equations. Cong and Liu in \cite{cl16} studied a degenerate $p$-Laplacian Keller-Segel model.

The rest of the paper is organized as follows. In Section \ref{sec:wass}, we give a brief introduction to Wasserstein-$p$ distances. The Benamou-Brenier characterization allows us to derive the Euler-Lagrange equations for the geodesics. We reveal the underlying Hamiltonian structure.  In Section \ref{sec:pEulerIncom}, we derive the Euler-Lagrange equations for the action represented by the Benamou-Brenier functional with the incompressibility constraint. The resulted equations are named as incompressible $p$-Euler equations. The structures are similar as the usual Euler equations but the momentum is replaced with the $p$-momentum that is nonlinear in the velocity. In Section \ref{sec:peulerfixed}, we reveal the variational structures of the $p$-Euler equations on a fixed domain, investigate the conservation of $p$-Hamiltonian and the streamfunction-vorticity formulation. In Section \ref{sec:pns}, we add a viscosity with $\gamma$-Laplacian to obtain the $p$-Navier-Stokes equations. This viscosity is a monotone function of the velocity and physically corresponds to shear-thinning or thickening effects. The special cases include $\gamma=2$ and $\gamma=p$. In Section \ref{sec:weaksol}, we show the global existence of weak solutions for $p$-Navier-Stokes equations with $p$-Laplacian viscosity.

\section{The Wasserstein-$p$ geodesics}\label{sec:wass}

In this section, we first give a brief introduction to the Wasserstein-$p$ distance and its relation to the Benamou-Brenier functional.
Then, we investigate the geodesics for Wasserstein-$p$ distances, which satisfy the pressureless $p$-Euler equations. Some underlying Hamiltonian structure will be discussed.

\subsection{Optimal transport and Wasserstein-$p$ distance}\label{subsec:ot}

Let $O\subset \mathbb{R}^d$ be a domain. Denote $\mathcal{P}(O)$ the set of probability measures on $O$. Let $\mu,\nu\in \mathcal{P}(O)$ and $c: O\times O \mapsto [0,\infty)$ be a cost function. The optimal transport problem is to optimize the following minimization:
\[
\min_{\gamma}\left\{\int_{O\times O} c \,d\gamma \Big| \gamma\in \Pi(\mu,\nu) \right\} ,
\]
where $\Pi(\mu, \nu)$ is the set of `transport plans', i.e. a joint measures on $O\times O$ so that the marginal measures are $\mu$ and $\nu$. If there is a map $T: O\to O$ such that $(I\times T)_{\#}\mu$ minimizes the target function, where $I$ is the identity map and
\[
(I\times T)_{\#}\mu (E):=\mu((I\times T)^{-1}(E)),~\forall E\subset O\times O,~\mathrm{measurable},
\]
 then $T$ is called an optimal transport map.

The Wasserstein-$p$ distance $W_p(\mu, \nu)$ for $p\ge 1$ is defined as \[
W_p(\mu, \nu)=\left(\inf_{\gamma\in \Pi(\mu,\nu)}\int_{O\times O} |x-y|^p d\gamma\right)^{1/p}.
\] 
In other words, $W_p^p(\mu, \nu)$ is the optimal transport cost with cost function $c(x,y)=|x-y|^p$. 

It has been shown in \cite[Chap. 5]{santambrogio15} that the Wasserstein-$p$ ($p\ge 1$) distance between two probability measures $\mu$ and $\nu$ is also given by
\begin{multline}\label{eq:Wp}
W_p^p(\mu, \nu)=\min\left\{\int_0^1\|v_t\|^p_{L^p(\rho)}dt: \partial_t\rho+\nabla\cdot(\rho v)=0, \rho_0=\mu, \rho_1=\nu \right\} \\
=\min\left\{\int_0^1 p\mathscr{B}_p(\rho, m)dt: \partial_t\rho+\nabla\cdot m=0, \rho|_{t=0}=\mu, \rho|_{t=0}=\nu \right\},
\end{multline}
where $\rho$ is a nonnegative measure and $m$ is a vector measure.  $\mathscr{B}_p$ is called the Benamou-Brenier functional. In the case $m\ll \rho$, $v=dm/d\rho$, 
\begin{gather}\label{eq:bbfunctional}
\mathscr{B}_p(\rho,m)=\int_{O}\frac{1}{p}|v|^p \rho(dx).
\end{gather}
As explained in \cite[Chap. 5]{santambrogio15}, $v$ can be understood as the particle velocity.

The minimizer for \eqref{eq:Wp} is called the Wasserstein-$p$ geodesic, which will not change if we consider the cost $c(x,y)=h(x-y)=\frac{1}{p}|x-y|^p$. Suppose $p>1$. The dual function of $h$, or Legendre transform of $h$, is \[
h^*(y)=\sup_{x}\left(x\cdot y-\frac{1}{p}|x|^p\right)=\frac{1}{q}|y|^q,
\]
where $1/p+1/q=1$. If $\mu$ is absolutely continuous with respect to Lebesgue measure, by the well-known result \cite{santambrogio15}, the optimal transport $T$ exists and is given by
 \[
x\mapsto T(x)=x-(\nabla h^*)(\nabla\psi (x))=x-|\nabla\psi(x)|^{q-2}\nabla\psi(x),
\]
for some $\psi$. The function $\psi$ is called Kantorovich potential. 
(From here on, if $\rho$ is absolutely continuous to the Lebesgue measure, and consequently, $\rho(dx)=\tilde{\rho}\, dx$ for some Lebesgue integrable function $\tilde{\rho}$, we may use $\rho$ to mean this function $\tilde{\rho}$ without much confusion.) The velocity of a particle is found to be constant and given by \[
v=T(x)-x=-|\nabla\psi(x)|^{q-2}\nabla\psi(x).
\]
With the Kantorovich potential, we are able to write out the $W_p$ distance as
 \[
W_p^p(\mu, \nu)=\int_0^1\int_{O}\rho |v|^p dx dt=\int_{O} |T(x)-x|^p\mu(dx)=\int_{O} |\nabla\psi(x)|^{q}\mu(dx).
\]

The velocity field $v(x, t=0)=T(x)-x$ is in general not curl free. However, an important observation is that the $p$-momentum
\begin{gather}\label{eq:pvel}
v_p:=T_p(v)=|v|^{p-2}v=-\nabla\psi
\end{gather}
is curl free. $v_p$ is the $(p-1)$-th signed power of $v$.  

\begin{rmk}
If $p=0$, $T_p$ is the Kelvin transform of $v$ about the unit sphere (in this paper, we consider $p>1$). If $p>1$, $T_p$ maps the inside of the unit sphere to the inside while maps outside to the outside. If $p\in (1, 2)$, the mapping is pushing points towards the unit sphere while in the $p>2$ case, the mapping is pushing points away from the unit sphere.
\end{rmk}

To find the optimal transport map from $\mu$ to $\nu$, one may want to solve for the Kantorovich potential $\psi$. Suppose $\mu$ and $\nu$ have densities $\rho_0$ and $\rho_1$ respectively, then by the fact
 \[
\rho_0=\rho_1 \det(J_T(x))
\]
where $J_T$ represent the Jacobian matrix of $T(x)$, we obtain the following equation for the potential $\psi$:
\[
\det\left(I-|\nabla\psi|^{q-2}(I+(q-2)\frac{\nabla\psi\otimes\nabla\psi}{|\nabla\psi|^2})\cdot\nabla\nabla\psi\right)\rho_1\Big(x-|\nabla\psi(x)|^{q-2}\nabla\psi(x)\Big)=\rho_0(x),
\]
where $\nabla\nabla\psi$ is the Hessian of $\psi$ and $\nabla\psi\otimes\nabla\psi$ is the tensor product. This equation can be regarded as the generalized version of the Monge-Ampe\'re equation.

Above we have used the tensor product $\nabla\psi\otimes\nabla\psi$. To be convenient for the discussion later, let us collect some notations for tensor analysis here. Let $a, b\in \mathbb{R}^d$ be vectors and $A, B$ be second order tensors (matrices).  $ab:=a\otimes b$ is a second order tensor, called the tensor product:
\begin{gather}\label{eq:tensorproduct}
(a\otimes b)_{ij}=(ab)_{ij}=a_ib_j.
\end{gather}
We also define the dot product as
\begin{gather}\label{eq:tensordot}
\begin{split}
& (a\cdot A)_i=\sum_{j=1}^d a_j A_{ji},~~(A\cdot a)_i=\sum_{i=1}^d A_{ij}a_j, \\
& (A\cdot B)_{ij}=\sum_{k}A_{ik}B_{kj},\\
&A:B=\tr(A\cdot B^T)=\sum_{ij}A_{ij}B_{ij}.
\end{split}
\end{gather}

\subsection{The Wasserstein-$p$ geodesics}
We now look at a particular optimal transport problem from the uniform distribution on a shape $\Omega_0\subset O$ to another probability measure $\nu$. We derive the equations for the corresponding flow formally. This formal derivation can be generalized to the case with incompressibility constraint to yield the incompressible $p$-Euler equations in Section \ref{sec:pEulerIncom}.

From here on, we will always assume 
\begin{gather}
p>1 .
\end{gather}

Consider an open set $\Omega_0\subset O$ and $\mu$ is the probability measure with density $\frac{1}{|\Omega_0|}1_{\Omega_0}$. Let $\nu$ be another probability measure supported in $\Omega_1\subset O$. Denote the optimal map by $T$. By \eqref{eq:Wp}, the Wasserstein-$p$ geodesic between $\mu$ and $\nu$ is the minimizer of \[
p\int_0^1\mathscr{B}_p(v, m) dt=:p\int_0^1 B_p(\rho, v) dt,
\]
where
\begin{gather}
B_p(\rho, v)=\int_O \frac{\rho}{p} |v|^p dx=\int_{\Omega_t}\frac{\rho}{p} |v|^p dx.
\end{gather}
Here, $\Omega_t$ is the support of $\rho(\cdot, t)$, and forms a sequence of shapes. The geodesic induces a flow of mass. In Section \ref{subsec:ot}, we know that the velocity $v$ for the particle in the optimal transport is constant. Now, we rederive this formally. 

 Assume the flow map is $Z(\cdot, t): \xi\mapsto x=Z(\xi, t)$, where $\xi$ is a Lagrangian coordinate.  Since $\rho$ can be interpreted as the density of particles \cite{santambrogio15}, we have 
\begin{gather}
\rho(x, t)|_{x=Z(\xi, t)}=\rho_0\det\left(\frac{\partial x}{\partial \xi}\right)^{-1},~\rho_0(\xi)=\frac{1}{|\Omega_0|}\chi(\Omega_0).
\end{gather}

Let us consider the cost function,
\begin{gather}\label{eq:actionbb}
\mathcal{A}=\int_0^1B_p(\rho, v) dt=\int_0^1 \frac{1}{p}\int_{\Omega_t} \rho |v|^p dx dt=\frac{1}{p|\Omega_0|}\int_0^1\int_{\Omega_0} |\dot{Z}|^p d\xi dt.
\end{gather}
Taking the variation with $\delta Z|_{t=0}=\delta Z|_{t=1}=0$:
\[
|\Omega_0|\delta \mathcal{A}=\int_0^1\int_{\Omega_0} |\dot{Z}|^{p-2}\dot{Z}\cdot \delta \dot{Z} d\xi dt
=-\int_0^1\int_{\Omega_0} \frac{\partial}{\partial t}\left(|\dot{Z}|^{p-2}\dot{Z}\right)\cdot \delta Z d\xi=0.
\]
Hence, we have
\begin{gather}
\left(I+(p-2)\frac{\dot{Z}}{|\dot{Z}|}\otimes \frac{\dot{Z}}{|\dot{Z}|}\right)\cdot\ddot{Z}=0.
\end{gather}
In Eulerian variables, we let
\begin{gather}
x=Z(\xi, t),~v(x, t)=\dot{Z}(\xi, t)|_{\xi=Z^{-1}(x, t)},
\end{gather}
and have
\begin{gather}
(I+(p-2)\hat{v}\hat{v})\cdot (\partial_tv+v\cdot\nabla v)=0 ,
\end{gather}
where $\hat{v}=v/|v|$ is the unit vector in the direction specified by $v$ and $\hat{v}\hat{v}=\hat{v}\otimes\hat{v}$ as in \eqref{eq:tensorproduct}.

Note that if $p>1$, $I+(p-2)\hat{v}\hat{v}$ is invertible and hence the minimizer satisfies \[
\partial_tv+v\cdot\nabla v=0,
\]
which confirms that particles in the geodesic have constant speed. Hence, for any $p>1$, any optimal flow satisfies the following pressureless compressible Euler equations
\begin{gather}\label{eq:presless}
\left\{
\begin{split}
& \partial_t\rho+\nabla\cdot(\rho v)=0,\\
& \partial_t(\rho v)+\nabla\cdot(\rho v\otimes v)=0.
\end{split}
\right.
\end{gather}
Clearly, for different $p$ values, the velocity fields should be different for the same initial and terminal distribution while they are described by the same equations \eqref{eq:presless}. Note that
\begin{gather}
\frac{d}{dt}\int_{\Omega_t}\frac{\rho}{p} |v|^p dx=0, ~\forall p>1.
\end{gather}
System \eqref{eq:presless} is so special that it conserves all the $p$-moment and it is why the geodesic for any $W_p$ distance satisfies the same system.

The $p$-momentum satisfies 
\begin{gather}
\partial_t v_p+v\cdot\nabla v_p=0,
\end{gather}
and, as we shall see later, the more intrinsic system of equations is the following pressureless $p$-Euler equations
\begin{gather}
\left\{
\begin{split}
& \partial_t\rho+\nabla\cdot(\rho v)=0,\\
& \partial_t(\rho v_p)+\nabla\cdot(\rho v\otimes v_p)=0.
\end{split}
\right.
\end{gather}

\subsection{Underlying Hamiltonian structure}

Consider the following action of a single particle
\begin{gather}
\mathcal{A}=\frac{1}{p}\int_0^1L(x(t), \dot{x}(t))dt=\frac{1}{p}\int_0^1 |\dot{x}|^p dt.
\end{gather}
The Hamiltonian is obtained by the Legendre transform
\begin{gather}
H(x, v_p)=\sup_{v}( v\cdot v_p-L(x, v(v_p, x)))=\frac{1}{q}|v_p|^q,~\frac{1}{p}+\frac{1}{q}=1.
\end{gather}
We can verify that the following Hamilton ODE holds
\begin{gather}
\left\{
\begin{split}
&\displaystyle\dot{x}=\frac{\partial H}{\partial v_p}=|v_p|^{q-2}v_p=v,\\
&\displaystyle\dot{v}_p=-\frac{\partial H}{\partial x}=0.
\end{split}\right.
\end{gather}

Physically speaking, there is no external force. Then, the Hamiltonian and Lagrangian are given by the `kinetical energy' only. $v_p$, as a conjugate variable of $v$, should be understood as the momentum and thus we call $v_p$ the $p$-momentum. 
In the Wasserstein geodesics, all the particles are independent and each of them satisfies this Hamilton ODE.

\section{Incompressible $p$-Euler equations}\label{sec:pEulerIncom}

In this section, we derive the Euler-Lagrange equations for the action represented by the Benamou-Brenier characterization of Wassertein-$p$ distance with the incompressibility constraint.  The derived equations are called $p$-Euler equations. We then study some elementary properties of this system of equations.

\subsection{The variational problem and the Euler-Lagrange equations}
From here on, without loss of generality, we will assume in the variational problem:
\begin{gather}
|\Omega_0|=1.
\end{gather}
Consider $\mu=\chi(\Omega_0)$ and $\nu=\chi(\Omega_1)$ where $|\Omega_1|=|\Omega_0|=1$. The optimal transport from $\mu$ to $\nu$ is the minimizer of \eqref{eq:actionbb}. Now, we instead pursue the minimizer of \eqref{eq:actionbb} with the incompressibility constraint: $|\Omega_t|=|\Omega_0|=1$. In other words, we pursue the minimizer over the sequences of shapes with equal volume.  This motivates us to consider the functional
\begin{gather}\label{eq:actionincompressible}
\mathcal{A}=\int_0^1\int_{\Omega_0}\left( \frac{1}{p}|\dot{Z}|^p+\pi\left(\det\Big(\frac{\partial Z}{\partial \xi}\Big)-1\right) \right) d\xi dt
+\int_{\Omega_0}\lambda \left(\det\Big(\frac{\partial Z}{\partial \xi}\Big)-1\right)|_{t=1} d\xi,
\end{gather}
where $\pi$ and $\lambda$ are Lagrangian multipliers for the constraint that the map is incompressible. Note that we are only asking for $Z(\Omega_0, 1)=\Omega_1$ but $Z(\xi, 1)$ is free to change.

Taking variation $\delta Z$ such that $\delta Z|_{t=0}=0$ and $Z(\Omega_0, 1)=(Z+\delta Z)(\Omega_0, 1)$:
\[
\delta \mathcal{A}=\int_0^1\int_{\Omega_0}\left( |\dot{Z}|^{p-2}\dot{Z}\cdot \delta\dot{Z} 
+\pi\tr\left( \Big(\frac{\partial Z}{\partial \xi}\Big)^{-1}
\frac{\partial \delta Z}{\partial \xi}\right) \right) d\xi dt+\int_{\Omega_0}\lambda\tr\left(\Big(\frac{\partial Z}{\partial \xi}\Big)^{-1}
\frac{\partial \delta Z}{\partial \xi}\right)|_{t=1} d\xi.
\]
Recall that 
\[
v(Z(\xi, t), t)=\dot{Z}(\xi, t),
\]
and we introduce $\delta z$ as
\[
\delta z(Z(\xi, t), t)=\delta Z(\xi, t).
\]
Integration by parts, we have
\begin{multline*}
-\int_0^1\int_{\Omega_t}\delta z\cdot(\partial_t+v\cdot\nabla)(|v|^{p-2}v)\, dxdt
+\int_{\Omega_1}|v|^{p-2}v\cdot \delta z \,dx\\
-\int_0^1\int_{\Omega_t}\delta z\cdot\nabla \pi \, dSdt
+\int_0^1\int_{\partial\Omega_t}\pi \delta z\cdot n \,dSdt
-\int_{\Omega_1}\delta z\cdot\nabla\lambda \,dx+\int_{\partial\Omega_1}\lambda
\delta z\cdot n \,dS=0 .
\end{multline*}
By the fact that $Z(\Omega_0, 1)=\Omega_1$, $\delta z\cdot n=0$ on $\partial\Omega_1$ at $t=1$. As a result, we have the following equation:
\begin{gather}\label{eq:pEuler1}
(\partial_t+v\cdot\nabla)(|v|^{p-2}v)=-\nabla \pi,~x\in \Omega_t, 0<t<1,
\end{gather}
with conditions
\begin{gather*}
|v|^{p-2}v|_{t=1}=\nabla\lambda,\\
\pi=0, \ x\in\partial\Omega_t.
\end{gather*}
The $p$-momentum $v_p=|v|^{p-2}v$ has to be irrotational at $t=1$, which implies that \eqref{eq:pvel} is natural. 

\begin{rmk}
One may be tempted to write this system into the form of usual Euler equations by inverting $I+(p-2)\hat{v}\hat{v}$ (recall that $\hat{v}\hat{v}$ is the tensor product as in \eqref{eq:tensorproduct}):
\begin{gather*}
\left\{
\begin{split}
& \partial_tv+v\cdot\nabla v=-\left(I+\frac{2-p}{p-1}\hat{v}\hat{v}\right)\cdot(|v|^{2-p} \nabla \pi),\\
&\nabla\cdot v=0,
\end{split}
\right.
\end{gather*}
with the boundary condition
\[
\pi=0, ~x\in\partial\Omega_t.
\]
This form, however, is not convenient for us to study.
\end{rmk}

\subsection{The $p$-Euler equations}

In this subsection, we discuss the general properties of system \eqref{eq:pEuler1} in more detail. To be convenient, we write out the system of equations below for $x\in \Omega_t$:
\begin{gather}\label{eq:pEuler}
\left\{
\begin{split}
&\partial_t v_p+v\cdot\nabla v_p=-\nabla \pi,\\
& v_p=|v|^{p-2}v,\\
&\nabla\cdot v=0.
\end{split}
\right.
\end{gather}
Here, we are not specifying the boundary conditions at $\partial\Omega_t$ for general discussion. If we specify $\pi=0,~x\in\partial\Omega_t$, we have the free boundary problem while if we set $\Omega_t=\Omega$, and $v\cdot n=0$ (where $n$ is the outer normal of $\Omega$) for $x\in\partial \Omega$, we have the fixed domain problem. This system is called the (incompressible) $p$-Euler equations due to the similarities with the usual Euler equations.

Consider again the Lagrangian (with $\rho=1$ in $\Omega_t$):
\begin{gather}\label{eq:lag}
L=\int_{\Omega_t}\mathcal{L}(v)\,dx, ~\mathcal{L}(v)=B_p(1, v)=\frac{1}{p}|v|^p,
\end{gather}
and the action $\mathcal{A}=\int_0^1L\, dt$. The Legendre transform of $\mathcal{L}(v)$ is the Hamiltonian
\begin{gather}
\mathcal{H}(u)=\sup_{v}(u\cdot v-\mathcal{L}(v)) .
\end{gather}
The maximum happens when 
\begin{gather}\label{eq:pvel2}
u=|v|^{p-2}v=v_p,
\end{gather}
 and thus 
\begin{gather}
\mathcal{H}(u)=\frac{1}{q}|u|^q, ~\frac{1}{p}+\frac{1}{q}=1.
\end{gather}
Hence, we can introduce the $p$-Hamiltonian for System \eqref{eq:pEuler1}:
\begin{gather}\label{eq:hal}
H(v_p)=\int_{\Omega_t}\mathcal{H}(v_p)dx=\int_{\Omega_t}\frac{1}{q}|v_p|^q dx.
\end{gather}
It is easy to find that
\begin{gather}\label{eq:laghal}
pL(v)=qH(v_p),
\end{gather}
and that the Fr\'echet derivatives have the dual symmetry:
\begin{gather}\label{eq:dualsym}
\frac{\delta L}{\delta v}=v_p,~\frac{\delta H}{\delta v_p}=v.
\end{gather}

In the case $d=2, 3$, we define $p$-vorticity as the curl of $v_p$
\begin{gather}
\omega_p=\nabla\times v_p.
\end{gather}

\begin{pro}\label{pro:formulationofpEuler}
Suppose $(v, \pi)$ is smooth and satisfies the incompressible $p$-Euler equations \eqref{eq:pEuler} for $d=2,3$. Then:

(i).  The momentum equation can be written in conservation form as
\begin{gather}
\partial_t v_p+\nabla\cdot (v_p\otimes v+\pi I)=0.
\end{gather}

(ii). If the flow is steady, then the Bernoulli's law holds: $\pi+\frac{1}{q}|v_p|^q$ is constant along a streamline. In other words,
\begin{gather}
v\cdot\nabla\left(\pi+\frac{1}{q}|v_p|^q\right)=0.
\end{gather}

(iii). The $p$-vorticity satisfies
\begin{gather}
\partial_t\omega_p+v\cdot\nabla\omega_p-(\omega_p\cdot\nabla)v=0.
\end{gather}
Consequently, if the flow is irrotational ($\omega_p=0$) at some time, it is irrotational for all time. For irrotational flows, $v_p=\nabla\phi$ and $\phi$ satisfies the Bernoulli equations
\begin{gather}\label{eq:bernoulli}
\left\{
\begin{split}
&\partial_t\phi+\frac{1}{q}|\nabla \phi|^q+\pi=c(t),\\
&-\Delta_q\phi=-\nabla\cdot(|\nabla\phi|^{q-2}\nabla\phi)=0,
\end{split}
\right.
\end{gather}
where the function $c(t)$ can be picked arbitrarily.
\end{pro}

\begin{proof}
Statement (i) follows from $\nabla\cdot v=0$.

(ii). If $d=2,3$, we have the following vector identity:
\begin{gather}\label{eq:vectoridentity}
v\cdot\nabla v_p=(\nabla v_p)\cdot v-v\times (\nabla\times v_p)
=\frac{p-1}{p}\nabla|v|^p-v\times (\nabla\times v_p).
\end{gather}
Using \eqref{eq:vectoridentity}, first equation in \eqref{eq:pEuler} is reduced 
\[
\partial_t v_p+\omega_p\times v+\nabla\left(\frac{p-1}{p}|v|^p+\pi \right)=0.
\]
If the flow is steady, $\partial_t v_p=0$. Dotting both sides with $v$ and noticing $\frac{p-1}{p}=\frac{1}{q}$ and $|v|^p=|v_p|^q$, we obtain what is claimed.

(iii). Taking curl in the first equation of System \eqref{eq:pEuler}, and using \eqref{eq:vectoridentity}, we derive the equation for $\omega_p$. According to the equation that $w_p$ satisfies, if $w_p=0$ for some time $t$, then it is zero for all time.

If $w_p=0$ for all time, then $v_p=\nabla\phi$. The first equation of System \eqref{eq:pEuler} is rewritten as
\[
\nabla \left(\partial_t\phi+\frac{p-1}{p}|v|^p+\pi \right)=0.
\]
Since $|v|^p=|v_p|^q=|\nabla\phi|^q$ and $\frac{p-1}{p}=\frac{1}{q}$, we obtain the first equation in \eqref{eq:bernoulli}.
Since $v=|v_p|^{q-2}v_p=|\nabla\phi|^{q-2}\nabla\phi$ and $v$ is divergence free, the second equation follows.
\end{proof}

Proposition \ref{pro:formulationofpEuler} confirms that $v_p$ is the physical momentum.  As a corollary, we have
\begin{cor}
Suppose the minimizer of Problem \eqref{eq:actionincompressible} exists and it induces a diffeomorphism from $\Omega_0$ to $\Omega_t$. For $d=2,3$, $v_p$ is irrotational:
\begin{gather}
\nabla\times v_p=0,~\forall t\in[0,1].
\end{gather}
\end{cor}
\begin{proof}
In the case $d=2,3$, by Equations \eqref{eq:pEuler1}, we see that $v_p$ is irrotational at $t=1$. Then, applying the equation of $\omega_p$, we find that $\omega_p$ is zero for all time $t\in [0,1]$ since the flow map is a diffeomorphism. The claim follows.
\end{proof}

\subsection{Influence of Galilean transform}

Since the momentum is nonlinear in the velocity, the influence of Galilean transform on the fluid system could be sophisticated.
Consider the Galilean transform
\begin{gather}
\begin{split}
& y=x+x_0+wt,\\
& \tau=t,
\end{split}
\end{gather}
where $w$ is a constant.  We introduce the new velocity and $p$-momentum
\begin{gather}
u=v+w,~~ u_p=|v+w|^{p-2}(v+w).
\end{gather}
The $p$-Euler equations for the new variables are given by 
\begin{gather}
\left\{
\begin{split}
&\partial_{\tau} u_p+u\cdot\nabla_y u_p=-\nabla_y\theta,\\
&\nabla_y\cdot u=0.
\end{split}
\right.
\end{gather}
Clearly,
\begin{gather}
\nabla_y\cdot u=0 \Leftrightarrow \nabla\cdot v=0.
\end{gather}
Using the relation
\[
\partial_{\tau}=\partial_t-w\cdot\nabla, ~ \nabla_y=\nabla,
\] 
we find 
\begin{multline*}
\partial_{\tau} u_p+u\cdot\nabla_y u_p=(\partial_t+v\cdot\nabla)u_p=(\partial_t+v\cdot\nabla)v\cdot |u|^{p-2}(I+(p-2)\hat{u}\hat{u})\\
=(\partial_t v_p+v\cdot \nabla v_p) \cdot |v|^{2-p}(I+(p-2)\hat{v}\hat{v})^{-1}\cdot |u|^{p-2}(I+(p-2)\hat{u}\hat{u}) .
\end{multline*}
Hence, the `$p$-pressures' are related by 
\begin{gather}\label{eq:pressure}
\nabla\theta\cdot  |u|^{2-p}\left(I+\frac{2-p}{p-1}\hat{u}\hat{u}\right)=\nabla\pi\cdot |v|^{2-p}\left(I+\frac{2-p}{p-1}\hat{v}\hat{v}\right).
\end{gather}
Assuming the $p$-Euler equations are well-posed, we then conclude that the $p$-Euler equations are Galilean invariant with the pressures related by Equation \eqref{eq:pressure}. 

\section{Conservation in free-boundary $p$-Euler equations}\label{sub:symm}

We have seen that the minimizer of  Problem \eqref{eq:actionincompressible} satisfies the free boundary incompressible $p$-Euler equations for $x\in\Omega_t, 0\le t\le 1$:
\begin{gather}\label{eq:pEulerfreeb}
\left\{
\begin{split}
&\partial_t v_p+v\cdot\nabla v_p=-\nabla \pi,\\
& v_p=|v|^{p-2}v,\\
&\nabla\cdot v=0,
\end{split}
\right.
\end{gather}
with boundary condition
\begin{gather}\label{eq:freebcs}
\pi=0,~x\in\partial\Omega_t.
\end{gather}
In this section, we use the action and Noether's first theorem (see \cite{noether71,arnold13}) to reveal some conserved quatities for the free-boundary $p$-Euler equations. Noether's theorem states that a differentiable symmetry of the action for a system induces a conserved quantity. Suppose the Lagrangian is given by 
$L(t, Z, \dot{Z})$ where $Z$ is the dynamics, and $Q$ is the generator of some symmetry (in other words, the action $\int_0^1 L\, dt$ is invariant under the transform $Z\to e^{\epsilon Q}Z$ and $\dot{Z}\to e^{\epsilon Q}\dot{Z}$), then the integral of Noether current is conserved:
\[
\left\langle Q(Z), \frac{\delta L}{\delta \dot{Z}}\right\rangle-F,
\] 
where the pairing is in $L^2$ sense. If we have time shift symmetry, $Q=\frac{d}{dt}$ and $F=L$. Then, the conserved quantity is the Hamiltonian
\[
H=\left\langle \dot{Z},  \frac{\delta L}{\delta\dot{Z}} \right\rangle-L=const.
\]
If $Q$ is not transforming time, then $F=0$ and we have the conserved quantity as
\[
\left\langle Q(Z), \frac{\delta L}{\delta \dot{Z}}\right\rangle.
\]

Recall that in our problem, $Z$ is the flow map, and the action is given by 
\[
\mathcal{A}=\int_0^1\mathcal{L}(t, Z, \dot{Z})\, dt=\int_0^1\int_{\Omega_0}\frac{1}{p}|\dot{Z}|^p \,d\xi dt,\quad \mathcal{L}(t, Z, \dot{Z})=\int_{\Omega_0}\frac{1}{p}|\dot{Z}|^p\, d\xi.
\]
Let us check several examples of symmetry for $d=2,3$:

\begin{itemize}
\item $L$ is independent of time. Then, we have the conservation of Hamiltonian:
\[
\int_{\Omega_0} \dot{Z}\frac{\delta L}{\delta \dot{Z}} \,d\xi-L=\int_{\Omega_0} \frac{p-1}{p}|\dot{Z}|^p\,d\xi=\int_{\Omega}\frac{1}{q}|v_p|^q\,dx=H.
\]

\item  In the case of translation, $Q_i(Z)=e_i, i=1,2,\ldots, d$ where $e_i$ is the natural basis vector for $\mathbb{R}^d$. This then results in the conservation of total momentum $\int_{\Omega_0}e_i\cdot \frac{\delta L}{\delta \dot{Z}}\,d\xi$. Or, in other words, 
\[
\frac{d}{dt}M_p:=\frac{d}{dt}\int_{\Omega_t} v_p\,dx=0.
\]
\item For rotation, $Q_a(Z)=a\times Z$ for some vector $a$. The action is invariant under rotation and thus we have the conservation of
\[
\int_{\Omega_0}(a\times Z)\cdot \frac{\delta L}{\delta \dot{Z}}\, d\xi
=a\cdot\int_{\Omega_t}x\times v_p\, dx.
\]
By the arbitrariness of $a$, we obtain the conservation of angular momentum
\[
\frac{d}{dt}\int_{\Omega_t}x\times v_p\,dx=0.
\]

\item Here, we investigate the conservation of circulation.
Let $t_0\in (0, 1)$ and $Z(z, t;t_0)$ be the flow map from $t_0$ to $t$ with $z$ being the Lagrangian variable. Picking a  closed curve $s\mapsto \gamma(s)$ at $t_0$, where $s$ is the arclength parameter.
We let the particles on $\gamma$ circuit with distance $\epsilon$ (to be rigorous, we need a $\delta$ tube of $\gamma$ and let the particles in this tube circuit). The flow map then results in an operation $e^{\epsilon Q}$ on $Z$, where  $Q(Z)=(\frac{d}{ds}\gamma)\cdot \nabla_zZ$ for $z\in \gamma(s)$. $e^{\epsilon Q}$ is a symmetry and we have the conservation of:
\[
\int_{\gamma} \left((\frac{d}{ds}\gamma)\cdot \nabla_zZ\right)\cdot\frac{\delta L}{\delta\dot{Z}} \,ds=\int_{\tilde{\gamma}} v_p \cdot \tau\,ds=:\Gamma,
\] 
where $\tilde{\gamma}=Z(\gamma, t; t_0)$ is the material curve while $\tau=\frac{d}{ds}\tilde{\gamma}$ is the unit tangent vector. $\Gamma$ is called the circulation.

\item Note that $\omega_p$ is divergence free and it evolves like a material vector field. Let $Z(z,t; t_0)$ be the flow map from $t_0$ to $t$ and $z$ is the Lagrangian variable at $t_0$. $\omega_p$ at $t_0$ then induces a volume preserving map $z\mapsto e^{\epsilon Q}z$ such that $Q(z)=\omega_p|_{t=t_0}$. Consequently, $Q(Z)=\omega_p|_{t=t_0}\cdot\nabla_zZ(z, t; t_0)$. The action is unchanged under the operation $Z\mapsto e^{\epsilon Q}Z$. We then have the following conserved quantity, called helicity,
\[
\int_{\Omega_{t_0}}Q(Z)\cdot \frac{\delta L}{\delta \dot{Z}}\, dx
=\int_{\Omega_{t_0}}(\omega_p|_{t=t_0}\cdot \nabla_zZ)\cdot |\dot{Z}|^{p-2}\dot{Z}dz=\int_{\Omega_t}\omega_p\cdot v_p\, dx.
\]
Note that $J=(\nabla_zZ)^T$ and the material vector field satisfies $\omega_p(\cdot, t)=J\cdot \omega_p(\cdot, t_0)$.
\end{itemize}  
All these conservation relations can be verified directly.

For compressible $p$-Euler equations with free boundary (note that this system is not closed as there are $2+d$ unknowns but $1+d$ equations only)
\begin{gather}\label{eq:pEulerCom}
\left\{
\begin{split}
&\partial_t\rho+\nabla\cdot(\rho v)=0,\\
&\rho(\partial_t v_p+v\cdot\nabla v_p)=-\nabla \pi,\\
& v_p=|v|^{p-2}v,
\end{split}
\right.
\end{gather}
where the boundary condition is \eqref{eq:freebcs}. 
The conservation of momentum is still true:
\begin{pro}
Suppose the solution to the $p$-Euler equations \eqref{eq:pEulerCom} is smooth. Then, the total $p$-momentum
\begin{gather}
M=\int_{\Omega_t}m \,dx, ~m=\rho v_p,
\end{gather}
is a constant.
\end{pro}
\begin{proof}
Using the first two equations in \eqref{eq:pEulerCom}, we find that the momentum $m$ satisfies the following conservation law
\begin{gather}
\partial_tm+\nabla\cdot(m\otimes v+\pi I)=0.
\end{gather}

Applying the Reynold's transport theorem,
\[
\frac{d}{dt}\int_{\Omega_t}m \,dx=\int_{\Omega_t}\partial_tm+\nabla\cdot(m_p\otimes v)\,dx
=-\int_{\Omega_t}\nabla\cdot(\pi I)\,dx=-\int_{\partial\Omega_t}\pi n \,dS=0.
\]
Hence, $M$ is a constant.
\end{proof}

\section{Incompressible $p$-Euler equations in a fixed domain}\label{sec:peulerfixed}

In this section, we discuss some general formulation of the initial value problem of $p$-Euler equations in a fixed domain instead of free boundary, i.e. $\Omega_t=\Omega\subset \mathbb{R}^d$ ($d=2,3$) and $\rho=1$ for $x\in\Omega$. The rigorous study of these PDEs is left for future.

\subsection{The formulation}

The initial value problem of $p$-Euler equations ($p>1$) in the fixed domain $\Omega$ for $0<t<1$ and $x\in\Omega$ is given by:
\begin{gather}\label{eq:pEulerRd}
\left\{
\begin{split}
&\displaystyle\partial_t v_p+v\cdot\nabla v_p=-\nabla \pi,\\
&\displaystyle v_p=|v|^{p-2}v,\\
&\displaystyle\nabla\cdot v=0,\\
\end{split}
\right.
\end{gather}
with initial value
\begin{gather}
\displaystyle v(x, 0)=v_0(x).
\end{gather}
In the case $\Omega\neq \mathbb{R}^d$, the boundary condition we impose is the no-flux boundary condition
\begin{gather}\label{eq:noflux}
v\cdot n=0, (v_p\cdot n=0),~x\in\partial\Omega.
\end{gather}
If $\Omega$ is unbounded, we require the quantities to decay at infinity.

Note that in the system \eqref{eq:pEulerRd}, $\nabla\pi$ is there to ensure $\nabla\cdot v=0$. Indeed, using the relation between $v_p$ and $v$, we can formally find that
\[
\partial_t v+v\cdot\nabla v=-|v|^{2-p}\left(I+\frac{2-p}{p-1}\hat{v}\otimes \hat{v}\right)\cdot\nabla\pi
=-|v_p|^{q-2}(I+(q-2)\hat{v}_p\otimes\hat{v}_p)\cdot\nabla\pi.
\]
Then, $\pi$ can be determined by the elliptic equation
\begin{gather}
-\nabla\cdot(|v_p|^{q-2}A_q(v_p)\nabla\pi)=\tr(\nabla v\cdot\nabla v),
\end{gather}
where $A_q(v_p)=I+(q-2)\hat{v}_p\otimes\hat{v}_p$ is positive definite with smallest eigenvalue $q-1>0$. 
This process is like the Leray projection defined based on the Helmholtz decomposition (see \cite{galdi11}) for the usual Euler equations.

It is straightforward to check that the conservations of $p$-Hamiltonian, helicity and circulation still hold
\begin{gather}
\frac{dH}{dt}=0,~\frac{d}{dt}\int_{\Omega_t}\omega_p\cdot v_p \,dx=0,~\frac{d}{dt}\int_{\tilde{\gamma}}v_p\cdot\tau\,ds=0,
\end{gather}
but we do not have conservation of $p$-momentum and angular momentum due to the fixed boundary.

\subsection{2D incompressible $p$-Euler equations in a bounded domain}
Consider the $p$-Euler equations \eqref{eq:pEulerRd} on a fixed domain $\Omega\subset \mathbb{R}^2$ with no-flux boundary condition.

Since $\nabla\cdot v=0$, then we are able to find the stream function $\psi$ such that 
\begin{gather}
v=\nabla^{\perp}\psi=\langle \partial_y\psi, -\partial_x\psi\rangle.
\end{gather}
By the no-flux boundary condition, we are able to choose $\psi|_{\partial\Omega}=0$. $v_p$ is then written as 
\begin{gather}
v_p=|v|^{p-2}v=|\nabla\psi|^{p-2}\nabla^{\perp}\psi ~~ \Rightarrow~~
\omega_p=\nabla\times v_p=\mathrm{div}(v_p^{\perp})=-\nabla\cdot(|\nabla\psi|^{p-2}\nabla\psi),
\end{gather}
where $a^{\perp}=\langle a_2, -a_1\rangle$. In other words, $p$-vorticity is the $p$-Laplacian of the stream function. 

Overall we have the following vorticity-streamfunction formulation
\begin{gather}\label{eq:pEulerRd2}
\left\{
\begin{split}
&\partial_t\omega_p+v\cdot\nabla\omega_p=0,\\
&-\Delta_p\psi=\omega_p,\\
&v=\nabla^{\perp}\psi,
\end{split}
\right.
\end{gather}
with the boundary and initial conditions
\begin{gather}
\begin{split}
&\psi=0,~x\in\partial\Omega,\\
&\omega(x, 0)=\omega_0(x),~x\in\Omega.
\end{split}
\end{gather}

Note that we have the following relations
\begin{gather}
qH=\int_{\Omega}|v_p|^q \,dx=\int_{\Omega} |v|^p \,dx=\int_{\Omega} |\nabla\psi|^p dx=\int_{\Omega} \omega_p\psi \,dx .
\end{gather}

Writing $qH$ in terms of $\omega_p$ and $\psi$ is convenient for the vorticity-streamfunction formulation:
\begin{gather}\label{eq:vorticity}
\frac{dH}{dt}=\int_{\Omega} \psi \partial_t\omega_p\,dx.
\end{gather}
For checking, direct computation reveals
\begin{gather*}
\frac{d (qH)}{dt}=\int_{\Omega} \partial_t\psi\omega_p+\psi \partial_t\omega_p\,dx .
\end{gather*}
The first term $\int_{\Omega} \partial_t\psi\omega_p dx$ equals \[
\int_{\Omega} \partial_t\psi \omega_p \,dx=-\int_{\Omega} \partial_t\psi\nabla\cdot(|\nabla\psi|^{p-2}\nabla\psi) dx
=\frac{1}{p}\int_{\Omega} \partial_t(|\nabla\psi|^p) dx. 
\]
Hence, 
\[
\frac{d(qH)}{dt}-\frac{1}{p}\frac{d(qH)}{dt}=\int_{\Omega}\psi\partial_t\omega_p\, dx,
\]
and \eqref{eq:vorticity} follows.

\begin{rmk}\label{rmk:Eulermodelvar} 
Formulation \eqref{eq:pEulerRd2} inspires the following model:
\begin{gather}
\left\{
\begin{split}
&\partial_t\omega_p+v\cdot\nabla\omega_p=\nu \Delta_p\psi,\\
&-\Delta_p\psi=\omega_p,\\
&v=\nabla^{\perp}\psi.
\end{split}
\right.
\end{gather}
This yields \[
\frac{d H}{dt}=-\nu \int_{\Omega} |\nabla\psi|^p dx=-\nu q H.
\]
The energy function $H$ decays exponentially. This is like porous media model where Darcy's law holds.  Another related model is the vorticity-streamfunction formulation for the $p$-NS equations with $\gamma=2$ in Section \ref{subsec:streampns}.
\end{rmk}

\section{Incompressible $p$-Navier-Stokes equations in a fixed domain}\label{sec:pns}

\subsection{Equations and preliminary investigation}

We now add viscosity term to the $p$-Euler equations to obtain the $p$-Navier-Stokes equations. The diffusion added is represented by the $\gamma$-Laplacian of $v$, and physically it is reminiscent  of the shear thinning or the shear thickening effects for non-Newtonian fluids (\cite{beirao09,los69,leibeson83,breit2017}). The initial value problem of the $p$-Navier-Stokes equations are then given by
\begin{gather}\label{eq:pNS}
\left\{
\begin{split}
&\displaystyle\partial_t v_p+v\cdot\nabla v_p=-\nabla \pi+\nu \Delta_{\gamma} v,\\
&\displaystyle v_p=|v|^{p-2}v,\\
&\displaystyle\nabla\cdot v=0,\\
\end{split}
\right.
\end{gather}
with initial condition
\begin{gather}
v(x, 0)=v_0(x).
\end{gather}
Here $\gamma>1$ and 
\begin{gather}
\Delta_{\gamma}v=\nabla\cdot(|\nabla v|^{\gamma-2}\nabla v), ~|\nabla v|=\sqrt{\sum_{ij}(\partial_i v_j)^2}.
\end{gather}
The $\gamma$-Laplacian viscous term corresponds to fast diffusion  if $1<\gamma<2$ and slow diffusion if $\gamma>2$.

In the case $\Omega\neq \mathbb{R}^d$, we specify the Dirichlet boundary condition 
\begin{gather}\label{eq:dirichlet}
v=0, ~x\in\partial\Omega .
\end{gather}
 Note that we have second derivative in space and we need more boundary conditions compared with the one \eqref{eq:noflux} for $p$-Euler equations. In the case $\Omega$ is unbounded, we require the solutions to decay fast enough at infinity.
 
\begin{rmk}\label{rmk:nonnewtonian}
The name `$p$-Navier-Stokes equations' is reminiscent of the models for non-Newtonian fluids studied by Breit in \cite{breit2015,breit2017}. We call \eqref{eq:pNS} the `$p$-Navier-Stokes equations' due to the Wasserstein-$p$ distance behind.
Compared with our model here, the models in \cite{breit2015,breit2017} are the usual Navier-Stokes equations with the viscous term replaced by $\mathrm{div}(|\varepsilon|^{p-2}\varepsilon)$ where $\varepsilon=\frac{1}{2}(\nabla v+\nabla v^T)$:
\[
\partial_tv+\mathrm{div}(v\otimes v)=-\nabla\pi+\mathrm{div}(|\varepsilon|^{p-2}\varepsilon)+f.
\]
Sometimes, this can also be called `$p$-Navier-Stokes equations', but clearly they mean different things.
\end{rmk}

\subsubsection{The variational structure}

In terms of $H$, the viscous term can be interpretated as 
\begin{gather}
\Delta_{\gamma}v=\mathrm{div}\left(\left|\nabla\frac{\delta H}{\delta v_p}\right|^{\gamma-2}\nabla\frac{\delta H}{\delta v_p}\right) .
\end{gather}
This form is quite similar to the opposite of Wasserstein gradient of a functional though Wasserstein gradient is for functionals of probability measures instead of vector fields. Multiplying $v=\frac{\delta H}{\delta v_p}$ on both sides of the first equation in \eqref{eq:pNS} and integrating, we find 
\begin{gather}
\frac{d}{dt}H=-\nu\int_{\Omega} \left|\nabla\frac{\delta H(v_p)}{\delta v_p}\right|^{\gamma} dx=-\nu\int_{\Omega}|\nabla v|^{\gamma}dx .
\end{gather}

There are two interesting diffusions for the $p$-Navier-Stokes \eqref{eq:pNS}. If $\gamma=2$, we have the usual diffusion and we discuss this case in Section \ref{subsec:streampns}. If $\gamma=p$, it is not hard to find the dual symmetry for the {\it a priori} estimates:
\begin{gather*}
v\in L^{\infty}(0,T; L^p)\cap L^p(0,T; W^{1,p}), ~ v_p\in L^{\infty}(0,T; L^q)\cap L^q(0,T; W^{1,q}),
\end{gather*} 
and the corresponding mollified estimates are listed in Proposition \ref{pro:apriori} below, which are useful for our proof of existence of weak solutions in Section \ref{sec:weaksol}. 

\begin{rmk}
A more physical term for diffusion that rheologists use is $\mathrm{div}(|\varepsilon|^{p-2}\varepsilon)$ where $\varepsilon=\frac{1}{2}(\nabla v+\nabla v^T)$. With this term, we find
\[
\langle \mathrm{div}(|\varepsilon|^{p-2}\varepsilon), v\rangle=-\|\varepsilon\|_p^p.
\]
By inequalities of Korn's type \cite[Lemma 2.1]{beirao09}, one can bound $\|\nabla v\|_p$ by $\|\varepsilon\|_p$ for a general class of functions. Hence, similar energy a priori estimates still hold. We use the dissipating term $\Delta_{\gamma}v=\mathrm{div}(|\nabla v|^{\gamma-2}\nabla v)$ here because of the fact $v=\frac{\delta H}{\delta v_p}$ and the mathematical convenience, while we note that this form captures essentially the same nonlinearity. Equations with more physical term $\mathrm{div}(|\varepsilon|^{p-2}\varepsilon)$ are left for future.
\end{rmk}

\subsubsection{Scaling invariance}\label{subsec:scale}

If we take the scaling for $\Omega=\mathbb{R}^d$ as $x=\lambda \bar{x}$, $ t=\lambda^{\alpha+1}\bar{t}$, then by the physical meaning
\begin{gather}\label{eq:velscale}
v=\frac{dx}{dt}=\frac{1}{\lambda^{\alpha}}v_{\lambda} \Rightarrow v_{\lambda}=\lambda^{\alpha}v(\lambda
\bar{x}, \lambda^{\alpha+1}\bar{t}).
\end{gather}
With the scaling $\pi_{\lambda}=\lambda^{p\alpha}\pi(\lambda \bar{x}, \lambda^{\alpha+1}\bar{t})$, we find 
\begin{gather*}
\partial_{\bar{t}}(v_{\lambda})_p+v_{\lambda}\cdot\nabla_{\bar{x}}(v_{\lambda})_p+\nabla_{\bar{x}} \pi_{\lambda}
=\lambda^{\alpha p+1}(\partial_tv_p+v\cdot\nabla v_p+\nabla \pi) .
\end{gather*}
The viscosity term is scaled as 
\begin{gather*}
\Delta_{\gamma}v_{\lambda}=\lambda^{\alpha\gamma-\alpha+\gamma}\Delta_{\gamma} v .
\end{gather*}
Hence, the equation is scaling-invariant if 
\begin{gather}
\alpha=\frac{\gamma-1}{p+1-\gamma}.
\end{gather}
Naturally, we require
\begin{gather}
\gamma<p+1.
\end{gather}

\begin{rmk}
The scaling for conserved quantity $\rho(x, t)$ is different from that in \eqref{eq:velscale}. Consider $x=\lambda \bar{x}$ and $t=\lambda^{1/\beta}\bar{t}$. To satisfy the conservation of mass $\int\rho(x, t)dx=\int \lambda^{\delta}\rho(\lambda\bar{x}, \lambda^{1/\beta}\bar{t})d\bar{x}$, we find that $\delta=d$. Hence, compared with the velocity (Equation 
\eqref{eq:velscale}), the suitable scaling for $\rho$ is given by, 
\begin{gather}\label{eq:scaleconserv}
\rho_{\lambda}(\bar{x}, \bar{t})\Rightarrow \lambda^d \rho(\lambda\bar{x}, \lambda^{1/\beta}\bar{t}).
\end{gather}
\end{rmk}

\begin{rmk}
For the doubly degenerate diffusion equation $\rho_t=\Delta_p \rho^m$, the scaling for $\rho$ is Equation \eqref{eq:scaleconserv}. The self-similar solution (by choosing $\lambda=\bar{t}^{-\beta}$) is given by
\[
\rho(x, t)=\frac{1}{t^{d\beta}}U\left(\frac{x}{t^{\beta}}\right) .
\]
Inserting this form into the diffusion equation, we find that the critical index $\beta_c$  for the self-similar solution satisfies 
\[
\beta_c=\frac{1}{p+dmp-dm-d}.
\]
Suppose $m>\frac{d-p}{d(p-1)}$ so that $\beta_c>0$. Then, one can find the following fundamental solution with initial data $u(x, 0)=\delta(x)$ which is self-similar, called Barenblatt solution \cite{aac10, barenblatt52}
\begin{gather}
U(\xi)=\begin{cases}
A_1\exp(-\frac{|p-1|^2}{p}|\xi|^{p/(p-1)}),~m=\frac{1}{p-1},\\
[\frac{m(p-1)-1}{mp}]^{(p-1)/(m(p-1)-1)}(R^{p/(p-1)}-|\xi|^{p/(p-1)})_+^{(p-1)/(m(p-1)-1)},~m\neq \frac{1}{p-1},
\end{cases}
\end{gather}
where $a_+=\max(a, 0)$ and $A_1$ and $R$ are determined by the fact that $\int U d\xi=1$. Since $R<\infty$, the fundamental solution in the case $m\neq 1/(p-1)$ has a compact support. If $m=1$, compared with the usual diffusion $p=2$, the diffusion given by a general $p$-Laplacian has finite propagation, and does not provide much smoothing effect.
\end{rmk}

\subsection{Vorticity-streamfunction formulation for 2D incompressible $p$-NS with $\gamma=2$}\label{subsec:streampns}

Consider $\gamma=2$ and $\Omega\subset\mathbb{R}^2$.  We have the vorticity-streamfunction formulation for \eqref{eq:pNS}:
\begin{gather}
\left\{
\begin{split}
&\partial_t\omega_p+v\cdot\nabla\omega_p=-\nu \Delta^2\psi,\\
&-\Delta_p\psi=\omega_p,\\
&v=\nabla^{\perp}\psi.
\end{split}
\right.
\end{gather}
For this formulation, the following energy dissipating relation could be useful for the analysis: 
\begin{multline*}
\frac{dH}{dt}=\frac{1}{q}\frac{d}{dt}\int_{\Omega} |\nabla\psi|^p dx=\frac{1}{q}\frac{d}{dt}\int_{\Omega}\omega_p\psi \,dx=\int_{\Omega}\psi \partial_t\omega_p dx\\
=-\int_{\Omega}\psi v\cdot\nabla\omega_p dx-\nu\int_{\Omega}|\Delta\psi|^2 dx=-\nu \int_{\Omega} |\Delta \psi|^2 dx .
\end{multline*}
Note that $-\int_{\Omega}\psi v\cdot\nabla \omega_p dx
=\int_{\Omega}(\nabla\psi\cdot v)\omega_p dx=0$

\section{Existence of weak solutions for $p$-NS in $\mathbb{R}^d$ with $\gamma=p$}\label{sec:weaksol}

We show in this section the global existence of weak solutions of the $p$-NS equations \eqref{eq:pNS} in $\Omega=\mathbb{R}^d$  with $\gamma=p$  and
\begin{gather}
p\ge d\ge 2.
\end{gather}

Let us now collect some notations for convenience. Suppose $X, Y$ are two Banach spaces. The notation $C^{\infty}(A; Y)$ for $A\subset X$ represents the class of all the infinitely smooth functions $f: A\mapsto Y$, while $C_c^{\infty}(A; Y)$ represents all the smooth functions $f: A\mapsto Y$ but with compact supports. If the codomain $Y$ is clear from the context, we may simply use $C^{\infty}(A)$ or $C_c^{\infty}(A)$ for short. Similar notations are adopted for $L^p$ spaces and Sobolev spaces $W^{1,p}$. If $A=[0, T]\subset \mathbb{R}$, we may write $C^{\infty}(0, T; Y)$ or $L^p(0, T; Y)$ for clarity.

Let $V$ be a Banach space. We use $V'$ to represent the dual space of $V$. Let $f\in V'$ and $u\in V$, the pairing between $f$ and $u$ is denoted as 
\begin{gather}
\langle f, u\rangle :=f(u).
\end{gather}

We need to define the following time distributional derivative with initial data, which is suitable for our definition of weak solutions:
\begin{definition}\label{def:weakderi}
We say $w\in (C_c^{\infty}[0, T))'$ is the time derivative of a function $f\in L_{loc}^1[0, T)$ with a given initial data $f_0$ if \[
\langle w, \varphi \rangle=-\langle  f-f_0, \varphi'\rangle, \forall \varphi\in C_c^{\infty}[0, T).
\]
We denote $\frac{d}{dt}f:=w$. We say $w\in (C_c^{\infty}(\mathbb{R}^d\times [0, T)))'$ is the time derivative of a function $f\in L_{loc}^1(\mathbb{R}^d\times[0, T))$ with a given initial data $f_0(x)\in L_{loc}^1(\mathbb{R}^d)$ if \[
\langle w, \varphi \rangle=-\langle f-f_0, \partial_t\varphi\rangle=-\langle f, \partial_t\varphi \rangle-\int_{\mathbb{R}^d}\varphi(x,0)f_0(x)dx,~~ \forall \varphi\in C_c^{\infty}(\mathbb{R}^d\times [0, T)).
\]
Below, we say $\frac{d}{dt}f\in W$ for a Banach space $W$ if $W$ can be continuously embedded into $(C_c^{\infty}[0,T))'$ and
 we can find $w\in W$ so that $w$ satisfies Definition \ref{def:weakderi}.
\end{definition}

\begin{rmk}
Note that if $f_0$ given is not consistent with the intrinsic initial value of $f$, $\frac{d}{dt}f$ contains some atom at $t=0$, while on the open interval $(0, T)$, $\frac{d}{dt}f$ agrees with the usual distributional derivative. For example, if we set $f(t)=1+t$ for $t>0$ while give $f_0=0$, then $\frac{df}{dt}=\delta(t)+1$.
\end{rmk}

\begin{rmk}
In the second part of the definition, for $\frac{d}{dt}f\in W$, a necessary condition is that every representative of $0\in W$ should be $0$ in $(C_c^{\infty}[0,T))'$.
\end{rmk}

A weak solution of \eqref{eq:pNS}  with initial value $v_0$ is a function such that the equations hold in the distribution sense where 
the time derivative is understood in the sense of Definition \ref{def:weakderi}. In particular, we have:
\begin{definition}\label{defi:weaksol}
We say $v\in L^{\infty}(0,T; L^p(\mathbb{R}^d))\cap L^{p}(0, T; W^{1,p}(\mathbb{R}^d))$ is a weak solution to the $p$-NS equations \eqref{eq:pNS} with initial data $v_0\in L^p(\mathbb{R}^d)$ on $[0, T)$, if it has time regularity in the sense
\begin{gather}\label{eq:timeregularityrequire}
\lim_{h\to 0^+}\int_0^{T-h}\|v(t+h)-v(t)\|_{L^p(\mathbb{R}^d)}^p dt=0,
\end{gather}
and $\forall \varphi \in C_c^{\infty}(\mathbb{R}^d\times [0, T); \mathbb{R}^d)$, $\nabla\cdot\varphi=0$, $\psi\in C_c^{\infty}(\mathbb{R}^d\times [0, T); \mathbb{R})$, we have
\begin{align}\label{eq:weakform}
\begin{split}
\int_0^T\int_{\mathbb{R}^d}v_p\cdot \partial_t\varphi\,dx dt&+\int_0^T\int_{\mathbb{R}^d}\nabla\varphi:vv_p dx \\
&-\nu\int_0^T\int_{\mathbb{R}^d}
\nabla\varphi:\nabla v|\nabla v|^{p-2}dxdt+\int_{\mathbb{R}^d}v_p(x, 0)\cdot \varphi(x, 0)dx=0,\\
\int_0^T\int_{\mathbb{R}^d}\nabla\psi\cdot v \,dxdt=0.
\end{split} 
\end{align}
If $v\in L_{loc}^{\infty}(0,\infty; L^p(\mathbb{R}^d))\cap L_{loc}^{p}(0, \infty; W^{1,p}(\mathbb{R}^d))$, $\forall T>0$, condition  \eqref{eq:timeregularityrequire} holds and the two integrals in \eqref{eq:weakform}  with $T$ replaced by $\infty$ hold for all $ \varphi \in C_c^{\infty}(\mathbb{R}^d\times [0, \infty); \mathbb{R}^d)$, $\nabla\cdot\varphi=0$, $\psi\in C_c^{\infty}(\mathbb{R}^d\times [0, \infty); \mathbb{R})$, we say $v$ is a global weak solution.
\end{definition}

Using the notation in \eqref{eq:tensordot}, we have
\begin{gather}
\nabla\varphi:vv_p=\sum_{ij}\partial_{i}\varphi_j v_i (v_p)_j,~\nabla\varphi:\nabla v=\sum_{ij}\partial_i\varphi_j \partial_i v_j.
\end{gather}

To prove the global existence, we first regularize the $p$-NS equations \eqref{eq:pNS} in Section \ref{subsec:regu} and provide some {\it a priori} estimates. Then, we prove a time shift estimate for the regularized solutions in Section \ref{subsec:compact}. Based on this time shift estimate and the {\it a priori} estimates, we use a variant of Aubin-Lions lemma to conclude  the compactness of the class of regularized solutions in $L^p(0, T; L^p(\Omega))$ for bounded set $\Omega$.  We then use the time regularity of the regularized solutions and the limit function to identify some weak limits in Section \ref{subsec:chi}. Finally, we conclude the global existence of weak solutions in Section \ref{subsec:existence}.

\subsection{Regularization and uniform estimates}\label{subsec:regu}
Pick $\zeta(x)\in C_c^{\infty}(\mathbb{R}^d)$ so that $\zeta\ge 0$ and $\int \zeta dx=1$. Define
\[
J_{\e}=\frac{1}{\e^d}\zeta\left(\frac{x}{\e}\right).
\]
We regularize the $p$-NS equations \eqref{eq:pNS} and initial data to be
\begin{gather}\label{eq:pNSReg}
\left\{
\begin{split}
&\displaystyle\partial_t v_p^\e+(J_{\e}*v^\e)\cdot\nabla v_p^\e=-\nabla \pi^{\e}+\nu \Delta_{p} v^\e+\epsilon\Delta v_p^\e ,\\
&\displaystyle v_p^\e=|v^\e|^{p-2}v^\e,\\
&\displaystyle\nabla\cdot v^\e=0,\\
&\displaystyle v^\e(x, 0)=J_{\e}*v_0(x) .
\end{split}
\right.
\end{gather}

\begin{lmm}\label{lmm:mollify}
If $u\in W^{k,p}(\mathbb{R}^d; \mathbb{R}^d)$ ($k\ge 0$) , then  there exists $C(\|u\|_{L^p},\e)>0$ such that
\begin{gather}\label{eq:mollilying}
\begin{split}
\|J_\e*u\|_{\infty}\le C(\|u\|_{L^p},\e),\\
\|J_\e*u\|_{W^{k,p}}\le \|u\|_{W^{k,p}} .
\end{split}
\end{gather}
If $k\ge 1$ and $\nabla\cdot u=0$, then $\nabla\cdot(J_{\e}*u)=0$.
\end{lmm}
\begin{proof}
Firstly, since $J_{\e}$ is a compact supported function, then $\|J_{\e}\|_q<\infty$ and thus 
 \[
\left|\int_{\mathbb{R}^d} J_{\e}(y)u(x-y) dy\right|\le \|u\|_p\|J_{\e}\|_q.
\]

The second inequality of \eqref{eq:mollilying} follows from Young's inequality for convolution.
Lastly, note $D^{\alpha}(J_{\e}*u)=J_{\e}*D^{\alpha}u$ where $D^{\alpha}$ is a derivative with some order. The last claim is then clear. 
\end{proof}

\begin{pro}\label{pro:apriori}
Suppose $v_0\in L^p(\mathbb{R}^d)$ and $p\ge 2$. If System \eqref{eq:pNSReg} has a strong solution on $(0, T)$ for some $T>0$, then 
\begin{gather}\label{eq:pnorm}
\frac{d}{dt}\frac{1}{q}\|v^{\e}(\cdot, t)\|_p^p=-\nu \int_{\mathbb{R}^d}|\nabla v^\e|^pdx-\epsilon\int_{\mathbb{R}^d} \nabla v^\e:\nabla v_p^{\e} dx \le -\nu \int_{\mathbb{R}^d}|\nabla v^\e|^pdx.
\end{gather}
In particular, the estimates for $v^{\e}$ and $v_p^{\e}$ have dual symmetry: there exists $C(p, \nu, T, \|v_0\|_{\mathbb{R}^d})>0$ independent of $\e$ such that
\begin{gather}\label{eq:uniformest}
\begin{split}
& \|v^\e\|_{L^{\infty}(0,T; L^p(\mathbb{R}^d))}=\|v_p^{\e}\|_{L^{\infty}(0,T; L^q(\mathbb{R}^d))} \le \|v_0\|_{L^p(\mathbb{R}^d)},\\
 &\max\left(\|v^{\e}\|_{L^p(0,T; W^{1,p}(\mathbb{R}^d))}, \|v_p^{\e}\|_{L^q(0,T; W^{1,q}(\mathbb{R}^d))}\right)\le C(p, \nu, T, \|v_0\|_{\mathbb{R}^d}).
 \end{split}
\end{gather} 
\end{pro}

\begin{proof}
For the strong solution on $(0, T)$, we dot $v^\e=|v_p^{\e}|^{q-2}v_p^{\e}$ on both sides of the first equation in \eqref{eq:pNSReg} and integrate on $x$. We have the following relations (all integration domains are $\mathbb{R}^d$):
\begin{gather}
\begin{split}
&\int v^\e \cdot \partial_t v_p^\e dx=\int |v_p^\e|^{q-2}v_p^\e\cdot \partial_t v_p^\e dx
=\frac{1}{q}\frac{d}{dt}\int |v_p^\e|^q dx=\frac{d}{dt}\frac{1}{q}\int |v^\e|^p dx,\\
&\int (J_{\e}*v^{\e})\cdot\nabla v_p^\e\cdot v^{\e} dx=\frac{1}{q}\int (J_{\e}*v^{\e})\cdot \nabla (|v_p^\e|^q) dx=0,\\
&\int v^\e\cdot\nabla\pi \,dx=0,\\
&\int v^\e\Delta_pv^\e dx=-\int |\nabla v^\e|^p dx, ~|\nabla v^\e|=\sqrt{\sum_{ij}(\partial_i v^\e_j)^2},\\
& \int v^\e\Delta v_p^\e dx=-\int \nabla v^\e:\nabla v_p^\e dx
=-\int |v^\e|^{p-2}\nabla v^\e: \nabla v^\e\cdot (I+(p-2)\hat{v}^\e\hat{v}^\e) dx.
\end{split}
\end{gather}
Note that $A=(I+(p-2)\hat{v}^\e\hat{v}^\e)$ is a positive definite matrix and \[
\nabla v^\e: \nabla v^\e\cdot (I+(p-2)\hat{v}^\e\hat{v}^\e)
=\sum_{ij} A_{ij}\sum_k \partial_k v_i^\e\partial_k v_j^\e \ge 0.
\]
Therefore, letting $H^{\e}=\frac{1}{q}\|v^\e\|_p^p=\frac{1}{q}\|v_p^{\e}\|_q^q$, we have
\begin{gather*}
\frac{dH^{\e}}{dt}=-\nu \int_{\mathbb{R}^d}|\nabla v^\e|^pdx-\epsilon\int_{\mathbb{R}^d} \nabla v^\e:(I+(p-2)\hat{v}^\e\hat{v}^\e)\cdot\nabla v^\e dx
\le -\nu \int_{\mathbb{R}^d}|\nabla v^\e|^pdx .
\end{gather*}
Consequently, we obtain the first inequality in \eqref{eq:uniformest}:
\[
\|v^{\e}\|_{L^{\infty}(0,T; L^p(\mathbb{R}^d))}^p\le \|v_0^{\e}\|_{L^p(\mathbb{R}^d)}^p\le \|v_0\|_{L^p(\mathbb{R}^d)}^p,
\]
and 
\begin{gather*}
\|v^{\e}\|_{L^p(0,T; W^{1,p}(\mathbb{R}^d))}^p\le C(p,\nu,T)\|v_0\|_{L^p(\mathbb{R}^d)}^p.
\end{gather*} 

Since $p\ge 2$, $1<q\le 2\le p$. If $p=2$, it is clear that $v_p^\e\in L^q(0,T; L^q)$. Assume $q<2<p$. Using H\"older's inequality,
\begin{multline*}
\int_{\mathbb{R}^d} |\nabla v_p^\e|^q dx=\int_{\mathbb{R}^d} ||v^\e|^{p-2}(I+(p-2)\hat{v}^\e\hat{v}^\e)\cdot\nabla v^\e|^q dx \\
\le C(p)\int_{\mathbb{R}^d} |v^\e|^{(p-2)q}|\nabla v^\e|^q dx 
\le  C(p, \|v\|_{L^{\infty}(L^p)})\left(\int_{\mathbb{R}^d}|\nabla v^\e|^p dx\right)^{q/p}.
\end{multline*}
Hence,
\begin{gather*}
\int_0^T\int_{\mathbb{R}^d} |\nabla v_p^\e|^q dx dt
\le C(p,  \|v\|_{L^{\infty}(L^p)})T^{(p-q)/p}\left(\int_0^T\|\nabla v^\e\|_p^p dt\right)^{q/p}.
\end{gather*}
The claim is then proved.
\end{proof}

\begin{cor}
The regularized system \eqref{eq:pNSReg} has a global strong solution on $[0,\infty)$.
\end{cor}
\begin{proof}
The local existence of strong solution is standard. Let $T>0$ be the largest time for the existence of strong solution of \eqref{eq:pNSReg}.  By Proposition \ref{pro:apriori}, on $(0, T)$, $\|v^{\e}(\cdot, t)\|_{L^p}\le \|v_0\|_{L^p(\mathbb{R}^d)}$. Hence, by Lemma \ref{lmm:mollify}, there exists $C(\epsilon, \|v_0\|_{L^p(\mathbb{R}^d)})$ such that
 \[
\|J_{\e}*v^{\e}\|_{\infty}\le C(\epsilon, \|v_0\|_{L^p(\mathbb{R}^d)}), \forall t\in(0, T).
\]
This implies that $T=\infty$.
\end{proof}

\subsection{The compactness of $v^\e$}\label{subsec:compact}

The following lemma is from \cite{damascelli98,cl16}. We copy down the proof here for convenience.
\begin{lmm}\label{lmm:basicineq}
Let $p>1$, then there exists $C(p)>0$ so that $\forall \eta_1, \eta_2 \in\mathbb{R}^d$, then \[
(|\eta_1|^{p-2}\eta_1-|\eta_2|^{p-2}\eta_2)\cdot(\eta_1-\eta_2)\ge C(p)(|\eta_1|+|\eta_2|)^{p-2}|\eta_1-\eta_2|^2.
\]
\end{lmm}
\begin{proof}
Let $A(\eta)=|\eta|^{p-2}\eta$. Fix $\eta_1\neq \eta_2$ (when they are equal, it is trivial). Without loss of generality, we assume $|\eta_2|\ge |\eta_1|$.

Denote $\xi(t)=(1-t)\eta_2+t\eta_1$.  Then, we have
\begin{multline*}
(A(\eta_1)-A(\eta_2))\cdot(\eta_1-\eta_2)=\int_0^1 (\eta_1-\eta_2)\cdot \nabla A(\xi(t))\cdot(\eta_1-\eta_2) dt \\
=\int_0^1  (\eta_1-\eta_2)\cdot |\xi|^{p-2}(I+(p-2)\hat{\xi}\hat{\xi})\cdot (\eta_1-\eta_2) dt
\ge \min(p-1, 1)|\eta_1-\eta_2|^2\int_0^1|\xi|^{p-2}dt.
\end{multline*}
If $p\in (1,2)$, $|\xi|\le |\eta_1|+|\eta_2|$ and the inequality is trivial. 

If $p\ge 2$, when $|\eta_2|\ge 2|\eta_1-\eta_2|$, then \[
|\xi|\ge ||\eta_2|-t|\eta_1-\eta_2||\ge \frac{|\eta_2|}{2}\ge \frac{|\eta_2|+|\eta_1|}{4}.
\]
When $|\eta_2|< 2|\eta_1-\eta_2|$, letting $t_0=|\eta_2|/|\eta_1-\eta_2|$, we have
\begin{multline*}
\int_0^1|\xi|^{p-2} dt\ge \int_0^1 ||\eta_2|-t|\eta_1-\eta_2||^{p-2}dt= \int_0^1 |\eta_1-\eta_2|^{p-2}|t-t_0|^{p-2}dt\\
\ge \frac{|\eta_2|^{p-2}}{2^{p-2}}\int_0^{1/2}z^{p-2}dz\ge C_2(p)(|\eta_1|+|\eta_2|)^{p-2}.
\end{multline*}
\end{proof}

Denote the time shift operator:
\begin{gather}
\tau_h v^\e(x, t)=v^\e(x, t+h).
\end{gather} 
Now we prove a crucial time shift estimate for $v^{\e}$:
\begin{lmm}\label{lmm:timeestimate}
Suppose $p\ge d\ge 2$. $\|\tau_hv^{\e}-v^\e\|_{L^p(0,T-h; L^p(\mathbb{R}^d))}\to 0$ uniformly in $\e<1$ as $h\to 0+$.
\end{lmm}

\begin{proof}
We have for any $t\le T-h$,
\begin{gather}\label{eq:vpshift}
\tau_hv_p^\e-v_p^\e+\int_t^{t+h}v^{\e}\cdot\nabla v_p^{\e} ds=\int_t^{t+h}-\nabla \pi^{\e} \,ds+\nu\int_t^{t+h}\Delta_p v^{\e} \,ds
+\e\int_t^{t+h}\Delta v^\e_p ds.
\end{gather}

We now dot both sides with $\tau_hv(t)^\e-v^\e(t)$ and integrate on $x$. Below, we show how each terms are estimated.

 For $\tau_hv_p^\e-v_p^\e$ in \eqref{eq:vpshift}, we have by Lemma \ref{lmm:basicineq},
\begin{multline*}
\int (\tau_hv^\e-v^\e)\cdot (\tau_hv_p^\e-v_p^\e) dx \ge
C(p)\int (|\tau_h v^\e|+|v^\e|)^{p-2}|\tau_h v^\e-v^\e(t)|^2 dx \ge C(p)\|\tau_h v^\e-v^\e\|_p^p.
\end{multline*}

Since $v^\e$ is divergence free, the integral for $(\tau_hv^{\e}(\cdot, t)-v^{\e}(\cdot, t))\cdot \nabla\pi (\cdot, s)$ is zero. 

 For the other terms, we show how to estimate $\tau_h v^\e(t)$ term. The estimates for $v^\e(t)$ are similar.

Consider the $p$-Laplacian term. By Young's inequality:
\begin{multline}
\int_t^{t+h}\int_{\mathbb{R}^d} \tau_h v^\e\cdot \Delta_p v^\e dx ds
=-\int_t^{t+h} \int_{\mathbb{R}^d}(\nabla \tau_hv^\e:\nabla v^\e)|\nabla v^\e|^{p-2}dx ds\\
\le \int_t^{t+h} \int_{\mathbb{R}^d} \frac{1}{2}(|\nabla \tau_hv^\e|^2|\nabla v^\e|^{p-2}+|\nabla v^\e|^p) dx ds
\le \int_t^{t+h}\left(\frac{1}{p}\|\nabla \tau_hv^\e(t)\|_p^p+(1-\frac{1}{p})\|\nabla v^{\e}(s)\|_p^p\right)ds\\
=\frac{h}{p}\|\nabla \tau_hv^\e(t)\|_p^p+(1-\frac{1}{p})\int_t^{t+h}\|\nabla v^{\e}(s)\|_p^pds .
\end{multline}
The regularization diffusion term is similar,
\begin{multline}
\int_t^{t+h}\int_{\mathbb{R}^d}\tau_hv^\e(t)\cdot\Delta v_p^\e dxds=-\int_t^{t+h}\int_{\mathbb{R}^d} \nabla\tau_hv^\e(t):\nabla v_p^\e dxds \\
\le \int_t^{t+h}\frac{1}{p}\|\nabla \tau_hv^\e(t)\|_p^p+\frac{1}{q}\|\nabla v_p^\e(s)\|_q^q ds
=\frac{h}{p}\|\nabla\tau_h v^\e(t)\|_p^p
+\frac{1}{q}\int_t^{t+h}\|\nabla v_p^\e(s)\|_q^q ds .
\end{multline}

Consider the transport term. Applying Gagliardo-Nirenberg inequality, we have
\[
\|v^{\e}\|_{2p}^{2p}\le C\|\nabla v^{\e}\|_p^{d}\|v^{\e}\|_p^{2p-d}\le C(\|v^{\e}\|_{L^{\infty}(0, T;L^p(\mathbb{R}^d))})\|\nabla v^{\e}\|_p^d.
\]
Therefore, 
\begin{multline}
-\int_t^{t+h}\int_{\mathbb{R}^d} \tau_h v^\e(t)v^\e(s)\cdot\nabla v^\e_p(s) dx ds\le \int_t^{t+h} \frac{1}{2p}\Big(\|\tau_hv^\e (t)\|_{2p}^{2p}+\|v^{\e}(s)\|_{2p}^{2p}\Big) \\
+\frac{1}{q}\|v_p^{\e}(s)\|_{W^{1,q}}^q ds
\le C\left(h\|\nabla \tau_hv^\e(t)\|_p^d+\int_t^{t+h}\|\nabla v^\e(s)\|_p^d+\|v_p^{\e}(s)\|_{W^{1,q}(\mathbb{R}^d)}^q ds\right) .
\end{multline}

Overall, we have
\begin{multline}
\|\tau_hv^\e-v^\e\|_{L^p(0,T-h;L^p(\mathbb{R}^d))}^p
\le C\Big((\nu+\epsilon)\int_0^{T-h}(\|\nabla\tau_h v^\e(t)\|_p^p+\|\nabla v^\e(t)\|_p^p) h\,dt\\
+(1+\nu+\e)\int_0^{T-h}\int_t^{t+h}\|\nabla v(s)\|_p^p+\|\nabla v_p^\e(s)\|_q^q \,ds dt \\
+\int_0^{T-h}h(\|\nabla\tau_hv^\e(t)\|_p^d+\|\nabla v^\e(t)\|_p^d)dt
+\int_0^{T-h}\int_t^{t+h}\|\nabla v^\e\|_p^d \,ds dt \Big)\\
\le Ch(1+\nu+\epsilon)\left(\|v^\e\|^p_{L^p(0, T; W^{1,p}(\mathbb{R}^d))}+\|v_p^\e\|^q_{L^q(0, T; W^{1,q}(\mathbb{R}^d))}\right)
+Ch\int_0^T\|\nabla v^\e\|_p^d ds.
\end{multline}
If $d\le p$, the last term is bounded above by \[
\left(\int_0^T\|\nabla v^\e\|_p^p dt\right)^{d/p}T^{(p-d)/p}.
\]

Hence, we have
\begin{gather}
\|\tau_hv^\e-v^\e\|_{L^p(0,T-h;L^p(\mathbb{R}^d))}^p\le C h ,
\end{gather}
where $C$ is uniform for $\e<1$.
\end{proof}

The following lemma is a variant of the traditional Aubin-Lions lemma (\cite{cl12, maitre03, simon86}):
\begin{lmm}\label{lmm:aubinlions}
Let $X, B$ be Banach spaces, $p\in [1,\infty)$ and $\mathscr{B}: X \mapsto B$ is a compact mapping (possibly nonlinear). Assume $V$ is bounded subset of $L_{loc}^1(0, T; X)$
such that $U=\mathscr{B}(V)\subset L^p(0, T; B)$ and 
\begin{itemize}
\item $U$ is bounded in $L_{loc}^1(0, T; B)$,
\item $\|\tau_h u-u \|_{L^p(0, T-h; B)}\to 0$ as $h\to 0+$ uniformly for $u\in U$.
\end{itemize}
Then, $U$ is relatively compact in $L^p(0, T; B)$.
\end{lmm}
Here, $L^1_{loc}(0, T ; X)$ means the set of functions $f$ such that $\forall 0<t_1<t_2<T, f\in L^1(t_1, t_2; X)$, equipped with the semi-norms $\|f\|_{L^1(t_1, t_2;X)}$. A subset $U$ of $L^1_{loc}(0, T ; X)$
 is called bounded, if $\forall 0<t_1<t_2<T$, $U$ is bounded in $L^1(t_1, t_2; X)$.

\begin{pro}\label{pro:subsequenceconv}
There exists a subsequence $\e_k$, $v\in L^{\infty}(0, T; L^p(\mathbb{R}^d))\cap L^p(0,T; W^{1,p}(\mathbb{R}^d))$ and $\chi\in L^{q}(0,T; L^{q}(\mathbb{R}^d))$, such that as $k\to\infty$, 
\begin{gather}
\begin{split}
& v^{\e_k}\to v ~a.e, ~\mathrm{in}~\mathbb{R}^d\times[0, T), ~\mathrm{strongly~in}~L^p(0,T; L_{loc}^p(\mathbb{R}^d)), \\
& J_{\e_k}*v^{\e_k}\to v,~\mathrm{strongly~in}~L^p(0,T; L_{loc}^p(\mathbb{R}^d)), \\
& v_p^{\e_k} \to |v|^{p-2}v=:v_p, ~\mathrm{strongly~in}~L^q(0,T; L_{loc}^q(\mathbb{R}^d)) ,\\
& v_p^{\e_k}(x, 0) \to v_p(x, 0), ~\mathrm{strongly~in}~L^q(\mathbb{R}^d), \\
& \nabla v^{\e_k}\rightharpoonup \nabla v, ~\mathrm{weakly~in}~L^p(0, T; L^p(\mathbb{R}^d)),\\
& |\nabla v^{\e_k}|^{p-2}\nabla v^{\e_k} \rightharpoonup \chi, ~\mathrm{weakly~in}~L^{q}(0,T; L^{q}(\mathbb{R}^d)). 
\end{split}
\end{gather}
\end{pro}

\begin{proof}
We now apply Lemma \ref{lmm:aubinlions}. The mapping $\mathscr{B}: W^{1,p}(\Omega)\mapsto L^p(\Omega)$ in Lemma \ref{lmm:aubinlions}  is chosen as the identity map. For any $\Omega$ bounded, the embedding of  $W^{1,p}(\Omega)$ into $L^p(\Omega)$ is compact by the Rellich-Kondrachov theorem, and hence $\mathscr{B}$ is a compact operator. We now choose $V$ in Lemma \ref{lmm:aubinlions} as $\{v^{\e}\}$, and hence $U=\mathscr{B}(V)=V=\{v^{\e}\}$.  $U$ is bounded in $L^p(0, T; L^p(\Omega))$ by Proposition \ref{pro:apriori} and the second condition of Lemma \ref{lmm:aubinlions} is verified by Lemma \ref{lmm:timeestimate}. As a result, $\{v^{\e}\}$ is relatively compact in $L^p(0, T; L^p(\Omega))$ for any bounded $\Omega$.

Let us take $\Omega_n\to \mathbb{R}^d$ where $\Omega_n$ is compact. Applying Lemma \ref{lmm:aubinlions},  $\{v^{\e}\}$ has a subsequence that converges to $\chi_1$ in $L^p(0,T; \Omega_1)$ . This sequence again has a subsequence that converges to $\chi_2$ in $L^p(0,T; \Omega_2)$. By the standard diagonal argument,  we can pick out a subsequence such that \[
v^{\e_k}\to \chi_n, ~in~L^p(0,T; L^p(\Omega_n)),~as~k\to\infty, \forall n\ge 1.
\]
Clearly, $v^{\e_k}\to \chi_n$ in $L^p(0,T; L^p(\Omega_m)), m\le n$ as well. Hence, $\chi_n$ agrees with $\chi_m$ on $\Omega_m$ for $m\le n$. Hence, $\chi_n\to v$ a.e. for a measurable function $v$ and $v$ agrees with $\chi_n$ on $\Omega_n$. As a result, $v^{\e_k}\to v$ in $L^p(0,T; \Omega_n)$, or
\[
v^{\e_k}\to v, ~\mathrm{strongly~in}~L^p(0,T; L_{loc}^p(\mathbb{R}^d)).
\] 

Since convergence in $L^p$ implies that there is a subsequence that converges almost everywhere, a diagonal argument confirms that there is a subsequence that converges a.e. to $v$ in $[0, T]\times \mathbb{R}^d$. Without relabeling, we still use $v^{\e_k}$ to denote this subsequence. 

By Fatou's lemma
\begin{gather*}
\int_{t_1}^{t_2}\int_{\mathbb{R}^d}|v|^p dx dt\le \liminf_{k\to\infty}\int_{t_1}^{t_2}\|v^{\e_k}\|_p^p dt \le (t_2-t_1)\|v_0\|_p^p.
\end{gather*}
This implies that $v(\cdot, t)\in L^p(\mathbb{R}^d)$ a.e.. and $\int_0^t\|v\|_p^p ds$ is Lipschitz continuous with the Lipschitz constant to be $\|v_0\|_p^p$. As a result, 
\[
v\in L^{\infty}(0,T; L^p(\mathbb{R}^d)), ~ \|v\|_{L^{\infty}(0,T;L^p)}\le \|v_0\|_p.
\]

Now, for a fixed $\Omega_n$, we can pick a bigger set $K$ such that $J_{\e_k}*(\chi(K)\varphi)=J_{\e_k}*\varphi$ on $\Omega_n$ for any locally integrable function when $\e_k$ is small enough.
\begin{multline*}
\|J_{\e_k}*v^{\e_k}-v\|_{L^p(0,T; L^p(\Omega_n))}
\le \|J_{\e_k}*v^{\e_k}-J_{\e_k}*v\|_{L^p(0,T; L^p(\Omega_n))}+ \|J_{\e_k}*v-v\|_{L^p(0,T; L^p(\mathbb{R}^d))}\\
\le   \|J_{\e_k}*(\chi(K)v^{\e_k})-J_{\e_k}*(\chi(K)v)\|_{L^p(0,T; L^p(\mathbb{R}^d))}+ \|J_{\e_k}*v-v\|_{L^p(0,T; L^p(\mathbb{R}^d))}.
\end{multline*}
We have $ \|J_{\e_k}*v-v\|_{L^p(0,T; L^p(\mathbb{R}^d))}\to 0$ by the mollifying.
For the first, we can easily bound it by
\[
\|\chi(K)v^{\e_k}-\chi(K) v\|_{L^p(0,T; \mathbb{R}^d)}
=\|v^{\e_k}-v\|_{L^p(0,T; K)}\to 0.
\]

For any compact set $\Omega_n$, we have 
\[
\|v_p^{\e_k}\|_{L^q(0,T;L^q(\Omega_n))}=\|v^{\e_k}\|_{L^p(0,T; L^p(\Omega_n))}^{p/q}\to \|v\|_{L^p(0,T; L^p(\Omega_n))}^{p/q}=\|v_p\|_{L^q(0,T;L^q(\Omega_n))}.
\]
By the a.e. convergence, we conclude that \[
v_p^{\e_k}\to v_p,~\text{strongly in}~L^q(0,T; L^q(\Omega_n)).
\]

Further, by the definition, $v^{\e_k}(x, 0)=J_{\e_k}*v_0$, we find that $v^{\e_k}\to v_0$ a.e. and in $L^p$ strongly. It is also easy to see that there is a subsequence without relabeling such that $v_p^{\e_k}(x, 0)\to v_p(x, 0)$ in $L^q(\mathbb{R}^d)$. 

Since $\nabla v^{\e_k}$ is bounded in $L^p(0,T; L^p(\mathbb{R}^d))$, the set is weakly pre-compact in $L^p(t_1, t_2; L^p(\mathbb{R}^d))$. With a diagonal argument again, we are able to pick a subsequence (without relabeling) such that we have the weak convergence $\nabla v^{\e_k}\rightharpoonup \zeta\in L^p(t_1, t_2; L^p(\mathbb{R}^d))$ for all $t_1, t_2\in \mathbb{Q}\cap [0, T]$, $t_1<t_2$.  We pick $\varphi\in C_c^{\infty}(\mathbb{R}^d\times [0, T))\subset L^q(t_1, t_2; L^q(\mathbb{R}^d))=(L^p(t_1, t_2; L^p(\mathbb{R}^d)))'$, $\forall 0\le t_1<t_2\le T$, $t_1, t_2\in\mathbb{Q}$. Then, 
\begin{multline*}
\int_{t_1}^{t_2}\int_{\mathbb{R}^d}\zeta:\varphi \,dx dt=\lim_{k\to\infty}\int_{t_1}^{t_2}\int_{\mathbb{R}^d}\nabla v^{\e_k}:\varphi \,dxdt\\
=-\lim_{k\to\infty}\int_{t_1}^{t_2}\int_{\mathbb{R}^d}v^{\e_k}\cdot (\nabla\cdot\varphi) dxdt
=-\int_{t_1}^{t_2}\int_{\mathbb{R}^d}v\cdot (\nabla\cdot\varphi) dxdt.
\end{multline*}
where the last equality is by the fact that $v^{\e_k}\to v$ in $L^p(0,T; L_{loc}^p)$. As a result, \[
\int_{\mathbb{R}^d}\zeta:\varphi \,dx=-\int_{\mathbb{R}^d}v\cdot (\nabla\cdot\varphi) dx,~ a.e. t\in [0, T] .
\]
This is true for any $\varphi$. Now, since $C_c^{\infty}(\mathbb{R}^d\times [0, T])$ is separable with the topology of test function, we take a countable dense set $\{\varphi_n\}\subset C_c^{\infty}(\mathbb{R}^d\times [0, T])$. Then, there exists a set measure zero, $\Gamma$, such that  
\[
\int_{\mathbb{R}^d}\zeta:\varphi_n dx=-\int_{\mathbb{R}^d}v\cdot (\nabla\cdot\varphi_n) dx,~ \forall  t\in [0, T]\setminus \Gamma, \forall n.
\]
As a result, this equation is actually true $\forall \varphi\in C_c^{\infty}$ and $t\in [0, T]\setminus \Gamma$. Hence, 
\[
\nabla v=\zeta, ~\forall t\in [0, T]\setminus \Gamma.
\]
This then implies \[
v\in L^p(0, T; W^{1,p}(\mathbb{R}^d)).
\]

Lastly, we notice \[
\||\nabla v^{\e_k}|^{p-2}\nabla v^{\e_k}\|^{q}_{L^{q}(0, T; L^{q}(\mathbb{R}^d))}
=\int_0^T\int_{\mathbb{R}^d}||\nabla v^{\e_k}|^{p-2}\nabla v^{\e_k}|^{q}dx dt
=\int_0^T\int_{\mathbb{R}^d}|\nabla v^{\e_k}|^p dx dt<C .
\]
We again use the weak compactness and then done.
\end{proof}

\subsection{Time regularity and identifying $\chi$}\label{subsec:chi}

We now show the time regularity of some quantities and  identify $\chi$ with $|\nabla v|^{p-2}\nabla v$.
The essential ideas are the same as those in \cite{mm12} and \cite{cl16}. The time regularity estimates in this section are used to conclude Equation \eqref{eq:secondtime}, so that we can conclude Lemma \ref{lmm:importantlmm}. Together the Minty-Browder criterion (Lemma \ref{lmm:weakmono}), we identify $\chi$ in Proposition \ref{pro:chi}.

For the convenience of discussion, we introduce
\begin{gather}
\begin{split}
& V=\{v\in L^p(0, T; W^{1,p}(\mathbb{R}^d)), ~ \nabla \cdot v=0\}.
\end{split}
\end{gather}
Since $L^p(0, T; W^{1,p}(\mathbb{R}^d))$ is reflexive when $p\in (1,\infty)$, we conclude that $V$ is reflexive since it is a closed subspace of $L^p(0, T; W^{1,p}(\mathbb{R}^d))$.

Let $X$ be reflexive. An operator $G: X\rightarrow X'$ is called monotone
if $\forall v_1, v_2\in X$, \[
\langle Gv_1-Gv_2, v_1-v_2 \rangle \ge 0.
\]
By Lemma \ref{lmm:basicineq}, $-\Delta_p$ is clearly monotone. 
Now, fix $\phi\in C_c^{\infty}[0, T), \phi\ge 0$. Combining the result in \cite{mm12} and the fact that $|\nabla v^{\e_k}|^{p-2}\nabla v^{\e_k}\rightharpoonup \chi$ in $L^q(0,T; L^q(\mathbb{R}^d))$, we have
\begin{lmm}\label{lmm:boundedofpLap}
$-\Delta_p: L^p(0,T; \mathbb{R}^d)\to L^q(0,T; W^{-1,q}(\mathbb{R}^d))$ is a bounded, monotone operator and $\forall v_1, v_2\in L^p(0,T; \mathbb{R}^d)$, the mapping $\lambda\mapsto \langle \Delta_p(v_1+\lambda v_2), v_2\rangle$ is continuous.
There is a subsequence without relabeling that \[
\Delta_p v^{\e_k}\rightharpoonup \theta,~\mathrm{in}~ L^q(0,T; W^{-1,q}(\mathbb{R}^d)).
\]
and therefore also in $V'$, where $\theta$ is given by \[
\langle \theta, u\rangle =-\int_0^T \chi:\nabla u \,dx dt,~ u\in L^p(0, T; W^{1,p}(\mathbb{R}^d)) .
\]
\end{lmm}
Recall that $ \chi: \nabla u =\sum_{ij} \chi_{ij}\partial_i u_j$ by Equation \eqref{eq:tensordot}.  

We now consider the weak convergence of time derivatives. First consider the time derivative of $H$:
\begin{lmm}\label{lmm:reguH}
Let $H^{\e}(t)=\frac{1}{p}\int_{\Omega}|v^{\e}|^p dx$ and $\mathcal{M}[0, T]$ denote the set of Radon measures on $[0, T]$. Then, $-\frac{d}{dt}H\in \mathcal{M}[0, T]$ is nonnegative and
\begin{gather}
-\frac{d}{dt}H^{\e_k}\overset{*}{\rightharpoonup} -\frac{d}{dt}H, ~\mathrm{in}~\mathcal{M}[0, T].
\end{gather}
where the time derivative is understood as in Definition \ref{def:weakderi}.
\end{lmm}

\begin{proof}
By \eqref{eq:pnorm}, $-\frac{d}{dt}H^{\e_k}\in L^1[0, T]$ nonnegative. $L^1[0, T]\subset \mathcal{M}[0,T]=(C[0, T])'$ while $C[0, T]$ is separable. By Banach-Alaoglu theorem (any bounded set in the dual of a separable space is weak star precompact), there is a subsequence such that \[
-\frac{d}{dt}H^{\e_k}\overset{*}{\rightharpoonup}\alpha,~\text{in}~ \mathcal{M}[0, T] .
\] 

$\forall \phi\in C_c^{\infty}[0, T)$, \[
\langle \alpha, \phi \rangle =\lim_{k\to\infty}\langle  -\frac{d}{dt}H^{\e_k}, \phi \rangle
=\lim_{k\to\infty}\langle H^{\e_k}, \phi'(t) \rangle+\phi(0)H(0)=\langle H, \phi'(t) \rangle+\phi(0)H(0).
\]
Hence, $\alpha=-H'(t)$ in distribution sense as in Definition \ref{def:weakderi}.
\end{proof}

We now move onto the time regularity of $v_p$. We need the projection operator $\mathcal{P}$ (called Leray projection in some references) that maps a field to a divergence free field. Let $\Omega$ be a simply connected domain. Let $F_p(\Omega)$ be the completion of $\{w\in C_c^{\infty}(\Omega): \nabla\cdot w=0\}$ under the $L^p(\Omega)$ norm, and $G_p(\Omega)=\{w\in L^p(\Omega): \exists \alpha\in W_{loc}^{1,p}(\Omega), w=\nabla \alpha \}$. The following combines Theorem III.1.2 in \cite{galdi11} and relevant discussion there:
\begin{lmm}\label{lmm:hw}
Let $d\ge 2$ and $p\in (1,\infty)$. If $\Omega=\mathbb{R}^d$, a half-space or a bounded domain with $C^2$ boundary, then the Helmoholtz-Weyl decomposition holds:
\[
L^p(\Omega)=F_p(\Omega)\oplus G_p(\Omega).
\]
This then defines a projection operator $\mathcal{P}: L^p(\Omega)\to F_p(\Omega)$.
$\forall w\in L^p(\Omega)$,  there is a constant $C(p,\Omega)$ such that \[
\|\mathcal{P}w\|_p\le C(p,\Omega)\|w\|_p.
\]
\end{lmm}
\begin{rmk}
In the case $\Omega$ is bounded with smooth boundary, though the completion of $C_c^{\infty}(\Omega)$ under $L^p(\Omega)$ is the whole $L^p(\Omega)$ with any boundary conditions, the element $w\in F_p(\Omega)$ satisfies the no-flux boundary condition $w\cdot \b{n}=0$ in the weak sense. To see this, pick $\nabla\varphi\in G_q(\Omega)$ with $1/p+1/q=1$. Pick $w_n\in C_c^{\infty}(\Omega), \nabla\cdot w_n=0$ and $w_n\to w$ in $L^p(\Omega)$.  Clearly, $\int_{\Omega}w_n\cdot \nabla\varphi dx=0$.
Taking $n\to\infty$, we have $\int_{\Omega}w\cdot \nabla\varphi dx=0$. By the arbitrariness of $\nabla\varphi$, we conclude that $w$ satisfies the no-flux condition. If $d=2$ and $\Omega$ is the unit disk,  though $( 2x, -2y)$ is divergence free (also curl free),
\[
( 2x, -2y) \notin F_p(\Omega).
\]
\end{rmk}
\begin{rmk}
For general unbounded domains and $p\neq 2$, the Helmoholtz-Weyl decomposition may not be valid. See \cite{fks07, mb86}.
\end{rmk}
In the case that $\Omega=\mathbb{R}^d$, Lemma \ref{lmm:hw} follows directly from the Calderon-Zygmund theory for singular integrals in harmonic analysis (\cite{cz52, stein16}). 

If the field is in $W^{1,p}(\mathbb{R}^d)$, we have similarly
\begin{lmm}\label{lmm:controlproj}
Let $\varphi\in W^{1,p}(\mathbb{R}^d)$, then there exists $C(d)>0$ such that,\[
\|\mathcal{P}\varphi\|_{W^{1,p}}\le C(d)\|\varphi\|_{W^{1,p}}.
\] 
\end{lmm}
\begin{proof}
Consider the Helmholtz-Weyl decomposition,\[
\varphi=\mathcal{P}\varphi+\nabla \phi.
\] 
Then, $\phi$ is determined up to a constant by the uniqueness of the decomposition. One of such $\phi$ is given by: \[
\phi_0=-\Phi_d*(\nabla\cdot\varphi) ,
\]
where $\Phi_d$ is the fundamental solution satisfying $-\Delta \Phi_d=\delta(x)$ in distribution sense. 
By the singular integral theory \cite{stein16}
\[
\|\mathcal{P}\varphi\|_{W^{1,p}}\le \|\varphi\|_{W^{1,p}}+\|\nabla \phi_0\|_{W^{1,p}}
\le \|\varphi\|_{W^{1,p}}+A_d\|\nabla\cdot\varphi\|_{L^p}+\|\nabla\phi_0\|_{L^p}\le C(d)\|\varphi\|_{W^{1,p}}.
\]
Note that $\|\nabla\phi_0\|_{L^p}\le C_1\|\varphi\|_{L^p}$ because
we know already the projection operator is bounded from 
$L^p$ to $L^p$.
\end{proof}

We now show the following time regularity results:
\begin{lmm}\label{lmm:timeregu}
There exists a subsequence of $v^{\e_k}$ without relabeling such that  
\begin{gather}
\partial_t\mathcal{P}v_p^{\e_k}\rightharpoonup \partial_t\mathcal{P}v_p~\text{weakly in }L^q(0,T, W^{-1,q}(\mathbb{R}^d)).
\end{gather}

Further, $\partial_t\mathcal{P}v_{p}+\nabla\cdot(v\otimes v_p)=-\nu\theta$ in $V'$ where $\theta$ is given in Lemma \ref{lmm:boundedofpLap}. In particular, $\forall w\in V$, $\phi\in C_c^{\infty}[0, T)$,
\begin{gather}\label{eq:secondtime}
\langle  \partial_t\mathcal{P}v_p, \phi w \rangle-\int_0^T\phi(t)\int_{\mathbb{R}^d}
\nabla v\otimes v_p:w\, dx dt
+\nu\int_0^T\phi\int_{\mathbb{R}^d}
\nabla \chi:w \,dxdt=0.
\end{gather}
\end{lmm}

\begin{proof}
We pick $\tilde{\varphi}\in L^p(0, T; W^{1,p}(\mathbb{R}^d))$ with $\|\tilde{\varphi}\|_{L^p(0,T; W^{1,p})}\le 1$, and denote $\varphi=\mathcal{P}\tilde{\varphi}\in V$. Hence, by Lemma \ref{lmm:controlproj},
\[
\|\varphi\|_{L^p(0,T; W^{1,p})}\le C(d).
\]
Since $\langle \partial_t\mathcal{P}v_p^{\e_k}, \tilde{\varphi} \rangle=\langle  \partial_tv_p^{\e_k}, \varphi \rangle$. We estimate
\begin{gather*}
\int_0^T\int_{\mathbb{R}^d}\nabla \varphi:(J_{\e_k}*v^{\e_k})\otimes v_p^{\e_k} dx dt
\le \int_0^T\int_{\mathbb{R}^d} \frac{|\nabla \varphi|^p}{p}
+\frac{|J*v^{\e_k}|^{2p}}{2p}+\frac{2p-3}{2p}|v_p^{\e_k}|^{\frac{2p}{2p-3}}dxdt.
\end{gather*}
The first term is trivially bounded and the second term is bounded as we did in the proof of Lemma \ref{lmm:timeestimate}. For the third term, we apply Gagliardo-Nirenberg inequality and have
\[
\|v_p^{\e_k}\|_{\frac{2p-3}{2p}}^{2p/(2p-3)}
\le C\|\nabla v_p^{\e_k}\|_q^{d/(2p-3)}\|v_p^{\e_k}\|_q^{(2p-d)/(2p-3)}.
\]
Since $\frac{d}{2p-3}\le q=\frac{p}{p-1}$ for $p\ge d\ge 2$,
then the term is bounded.

Therefore, $\partial_t\mathcal{P}v_p^{\e_k}=\nu\Delta_p v^{\e_k}+\e_k\Delta v_p^{\e_k}-J_{\e_k}*v^{\e_k}\cdot\nabla v_p^{\e_k}$ is bounded in $L^q(0,T; W^{-1,q})$ by Lemma \ref{lmm:boundedofpLap}.  Hence, there exists $\beta\in L^q(0,T; W^{-1,q})$, such that 
\begin{gather}
\partial_t\mathcal{P}v_p^{\e_k}\rightharpoonup \beta~\text{weakly in }L^q(0,T; W^{-1,q}(\mathbb{R}^d)).
\end{gather}
If we take $\varphi\in C_c^{\infty}(\mathbb{R}^d\times [0,T))$, we have
\begin{gather*}
\langle  \partial_t\mathcal{P}v_p^{\e_k}, \varphi \rangle=-\langle \mathcal{P}v_p^{\e_k}, \partial_t\varphi \rangle
-\int_{\mathbb{R}^d}\mathcal{P}v_p^{\e_k}(x, 0) \varphi(x, 0) dx \to 
-\langle \mathcal{P}v_p, \partial_t\varphi  \rangle
-\int_{\mathbb{R}^d}\mathcal{P}v_p(x, 0) \varphi(x, 0)\, dx.
\end{gather*}
Hence $\partial_t\mathcal{P}v_p=\beta\in L^q(0,T; W^{-1,q}(\mathbb{R}^d))$ and $\partial_t\mathcal{P}v_{p}+\nabla\cdot(v\otimes v_p)=-\nu\theta$ in $V'$ by Lemma \ref{lmm:boundedofpLap}.
\end{proof}
\begin{rmk}
If one has a uniform estimate of $\nabla\pi^{\e_k}$ in $(L^p(0, T; W^{1,p}))'=L^q(0,T; W^{-1,q})$, one can have the weak convergence of $\partial_tv_p^{\e_k}$ to $\partial_t v_p$ in $(L^p(0, T; W^{1,p}))'=L^q(0,T; W^{-1,q})$.
\end{rmk}

\begin{lmm}\label{lmm:importantlmm}
It holds that $\forall \phi\in C_c^{\infty}[0, T)$,
\begin{gather}\label{eq:weakidendity}
\lim_{k\to\infty}\int_0^T\phi(t)\int_{\mathbb{R}^d}|\nabla v^{\e_k}|^pdx dt=-\langle v, \phi \theta\rangle=
\int_0^T\phi(t)\int_{\mathbb{R}^d}\nabla v:\chi \,dx dt.
\end{gather}
\end{lmm}

\begin{proof}
In Equation \eqref{eq:secondtime}, we set $w=v$ and get \[
\langle \partial_t\mathcal{P}v_p, \phi v \rangle
-\int_0^T\phi(t)\int_{\mathbb{R}^d}v\cdot\nabla v\cdot v_p dx dt
+\nu\int_0^T\phi(t)\int_{\mathbb{R}^d}\nabla v:\chi \,dx dt=0.
\]

Since $v\in L^{\infty}(0, T; L^p)\cap L^p(0,T; W^{1,p})$ and $v_p\in L^{\infty}(0, T; L^q)\cap L^q(0,T; W^{1,q})$, we find $\forall \varphi\in C_c^{\infty}(\mathbb{R}^d\times[0, T))$ and $\forall \delta>0$,
\[
\left\langle  \nabla\cdot(v*J_{\delta} \frac{1}{p}|v*J_{\delta}|^p), \varphi \right\rangle
=\left\langle v*J_{\delta}\cdot\nabla (v*J_{\delta})\cdot (v*J_{\delta})_p, \varphi \right\rangle .
\]
Taking $\delta\to 0$, we find that, \[
\left\langle \frac{1}{p}\nabla\cdot(v |v|^p), \varphi \right\rangle=
\left\langle v\cdot\nabla v\cdot v_p, \varphi \right\rangle.
\]
Note that $v\cdot\nabla v\cdot v_p$ and $v|v|^p$ are integrable on $\mathbb{R}^d\times [0, T)$ by similar argument in Lemma \ref{lmm:timeregu}.
Take the test function $\varphi=\phi(t)\zeta_n(x)$ where $\zeta_n$ is $1$ in $\Omega_n$. Sending $n\to\infty$, we find that $\int_0^T\phi(t)\int_{\mathbb{R}^d}v\cdot\nabla v\cdot v_p dxdt=0$. Hence, 
\begin{gather}\label{eq:aux1}
\langle \partial_t\mathcal{P}v_p, \phi v \rangle+\nu\int_0^T\phi(t)\int_{\mathbb{R}^d}\nabla v:\chi \,dx dt=0.
\end{gather}

In Equation \eqref{eq:pNSReg}, we multiply $\phi(t)v^{\e_k}$ and integrate, 
\begin{gather*}
\left\langle \frac{d}{dt}H^{\e_k}, \phi \right\rangle=-\int_0^T\phi \nu \int_{\mathbb{R}^d}|\nabla v^{\e_k}|^pdx-\epsilon\int_0^T\phi(t)\int_{\mathbb{R}^d} \nabla v^{\e_k}:\nabla v_p^{\e_k} dx .
\end{gather*}
The transport term vanishes since $\int_{\mathbb{R}^d}\nabla v^{\e_k}:v^{\e_k}v_p^{\e_k} dx=0$.

Letting $k\to\infty$, by Lemma \ref{lmm:reguH}, we have 
\begin{gather}\label{eq:aux2}
\left\langle \frac{d}{dt}H, \phi \right\rangle=-\nu \lim_{k\to\infty}\int_0^T\phi(t)\int_{\mathbb{R}^d}|\nabla v^{\e_k}|^pdx dt.
\end{gather}

Note $v\in V\cap L^{\infty}(0, T; L^p(\mathbb{R}^d))$ with $\partial_t\mathcal{P}|v|^{p-2}v\in L^q(0,T; W^{-1,q}(\mathbb{R}^d))$. We can pick $\varphi_n \to v$ strongly in $V\subset L^p(0,T; W^{1,p}(\mathbb{R}^d))$, $\varphi_n(x, 0)\to v_0(x)$ in $L^p(\mathbb{R}^d)$ such that $\partial_t|\varphi_n|^{p-2}\varphi_n$ are bounded in $V'$, which is achievable for example by mollification of $v$.   Let $\eta_n=|\varphi_n|^{p-2}\varphi_n$. It follows then
\begin{multline*}
\int_0^T \phi  \int_{\mathbb{R}^d} \varphi_n\cdot \partial_t\mathcal{P}\eta_n \,dx dt 
=-\int_0^T\phi' \frac{1}{q}\int_{\mathbb{R}^d} |\varphi_n|^p dx dt-\phi(0)\int_{\mathbb{R}^d}\frac{1}{q}|\varphi_n(x, 0)|^p dx \\
\to -\int_0^TH\phi'  dt-\phi(0)\frac{1}{q}\int_{\mathbb{R}^d} |v_0(x)|^p dx=\left\langle \frac{d}{dt}H, \phi \right\rangle.
\end{multline*}
Secondly, $\forall \varphi\in C_c^{\infty}(\mathbb{R}^d\times [0, T))\cap V$ which is dense in $V$ under the norm of $L^p(0,T; W^{1,p}(\mathbb{R}^d))$, 
\begin{multline*}
\langle \partial_t\mathcal{P}\eta_n, \varphi \rangle=-\langle  \mathcal{P}\eta_n, \varphi_t \rangle
-\int_{\mathbb{R}^d} \varphi(x,0)\eta_n(x,0)dx\to \\
-\langle  v_p, \varphi_t \rangle
-\int_{\mathbb{R}^d} \varphi(x,0)v_p(x,0)dx=\langle \partial_tv_p,  \varphi \rangle=\langle \partial_t\mathcal{P}v_p, \varphi \rangle.
\end{multline*}
Hence, $\partial_t\eta_n$ converges weak star to $\partial_t\mathcal{P}v_p$ in $V'$. Hence, $\int_0^T\phi\int_{\mathbb{R}^d}\varphi_n\cdot\partial_t\mathcal{P}\eta_n dxdt\to \int_0^T\phi\int_{\mathbb{R}^d}v\cdot\partial_t\mathcal{P}v_p dxdt$, and we have the chain rule:
\begin{gather}\label{eq:aux3}
\left\langle \frac{d}{dt}H, \phi \right\rangle_{[0, T]}=\langle  \partial_t\mathcal{P}v_p,  \phi v  \rangle.
\end{gather}
As a result, by \eqref{eq:aux1}, \eqref{eq:aux2} and \eqref{eq:aux3}, \eqref{eq:weakidendity} is shown.

\end{proof}

We will use the following Minty-Browder argument (\cite[Sec. 2.1]{barbu10}, \cite[Appendix]{mm12}):
\begin{lmm}\label{lmm:weakmono}
Suppose $X$ is a real reflexive Banach space. Let $G: X\mapsto X'$ be a nonlinear, bounded monotone operator such that $\forall v_1,v_2\in X$ the mapping $\lambda\mapsto \langle G(v_1+\lambda v_2), v_2 \rangle$ is continuous. If $w_n\rightharpoonup w$ in $X$ and $Gw_n\rightharpoonup \beta$ in $X'$ and \[
\limsup_{n\to\infty}\langle G w_n, w_n \rangle \le \langle  \beta, w \rangle,
\]
then $Gw=\beta$.
\end{lmm}
Using Lemma \ref{lmm:boundedofpLap}, Lemma \ref{lmm:importantlmm}-Lemma \ref{lmm:weakmono}, we are now able to conclude:
\begin{pro}\label{pro:chi}
With the settings and notations in Proposition  \ref{pro:subsequenceconv}, we have $\forall \varphi\in C_c^{\infty}(\mathbb{R}^d\times [0, T); \mathbb{R}^d)$ that
\[
\int_0^T\int_{\mathbb{R}^d}\chi: \nabla\varphi \, dxdt=\int_0^T\int_{\mathbb{R}^d}|\nabla v|^{p-2}\nabla v: \nabla\varphi \, dx dt.
\]
\end{pro}

\begin{proof}
Let $\phi(t)\in C_c^{\infty}[0, T)$ and $\phi(t)\ge 0$.
\[
\phi(t)\Delta_p: L^p(0,T; W^{1,p}(\mathbb{R}^d))\mapsto L^q(0, T; W^{-1,q}(\mathbb{R}^d)),
\]
satisfies all the properties for $\Delta_p$ listed in Lemma \ref{lmm:boundedofpLap}.

By Lemma \ref{lmm:importantlmm} and Lemma \ref{lmm:boundedofpLap},
\[
\lim_{k\to\infty} \langle \phi(t)\Delta_pv^{\e_k},  v^{\e_k} \rangle
=-\int_0^{T}\phi(t)\int_{\mathbb{R}^d}\chi: \nabla v \,dx dt
=\langle  \phi(t)\theta, v \rangle.
\]
By Lemma \ref{lmm:weakmono},
\[
\phi(t)\Delta_p v=\phi(t)\theta,~\textrm{in}~L^q(0,T; W^{-1,q}(\mathbb{R}^d))
\]
Then, it follows that for all $\varphi\in C_c^{\infty}(\mathbb{R}^d\times [0, T); \mathbb{R}^d)$, 
\[
\int_0^T\phi(t)\int_{\mathbb{R}^d}\nabla\varphi :|\nabla v|^{p-2}\nabla v \,dxdt=\int_0^T\phi(t)\int_{\mathbb{R}^d}\chi : \nabla\varphi \,dxdt.
\]
Clearly, for a fixed $\varphi$, we can pick $\phi$ such that 
\begin{gather*}
\int_0^T\phi(t)\int_{\mathbb{R}^d}\nabla\varphi :|\nabla v|^{p-2}\nabla v \,dxdt=\int_0^T\int_{\mathbb{R}^d}\nabla\varphi :|\nabla v|^{p-2}\nabla v \,dxdt,\\
\int_0^T\phi(t)\int_{\mathbb{R}^d}\chi : \nabla\varphi \,dxdt=\int_0^T\int_{\mathbb{R}^d}\chi: \nabla\varphi \,dxdt.
\end{gather*}
\end{proof}

\subsection{Existence of global weak solutions}\label{subsec:existence}

Now, we are able to claim that 
\begin{thm}
Consider $p\ge d\ge 2$ and $\Omega=\mathbb{R}^d$, the $p$-NS equations \eqref{eq:pNS} with $\gamma=p$ and $v_0\in L^p(\mathbb{R}^d)$ have a global weak solution.
\end{thm}

\begin{proof}

We first fix $T>0$ and show that the $p$-NS equations \eqref{eq:pNS} have a weak solution on $[0, T)$. 

Pick $\varphi$ and $\psi$ satisfying the conditions in Definition \ref{defi:weaksol}. Since $v^{\e_k}$ is divergence free, we have
\begin{gather*}
\int_0^T\int_{\mathbb{R}^d}\nabla\psi\cdot v^{\e_k} dxdt=0.
\end{gather*}
As $\e_k\to \infty$, since $\nabla\psi\in C_c^{\infty}(\mathbb{R}^d\times [0,T))$ and $v^{\e_k}\to v$ in $L^p(0,T; L_{loc}^p(\mathbb{R}^d))$, we find \[
\int_0^T\int_{\mathbb{R}^d}\nabla\psi\cdot v \,dx dt=0.
\]

Now, we dot $\varphi$ which is divergence free in Equation \eqref{eq:pNSReg} and integrate:
\begin{multline*}
\int_0^T\int_{\mathbb{R}^d}v_p^{\e_k}\cdot \varphi_t\,dx dt+\int_0^T\int_{\mathbb{R}^d}\nabla\varphi:J_{\e_k}*v^{\e_k}\otimes v_p^{\e_k} dx\\
-\nu\int_0^T\int_{\mathbb{R}^d}\nabla\varphi:\nabla v^{\e_k}|\nabla v^{\e_k}|^{p-2}dxdt
+\int_{\mathbb{R}^d}v_p^{\e_k}(x, 0)\cdot \varphi(x, 0)dx=
\e_k\int_0^T\int_{\mathbb{R}^d}\nabla\varphi\cdot\nabla v_p^{\e_k}dx.
\end{multline*}
Using the convergence in Proposition \ref{pro:subsequenceconv} and Proposition \ref{pro:chi}, we take the limit $k\to \infty$ and find that the first equation in Definition \ref{defi:weaksol} is satisfied as well. 
Hence, $v$ is a weak solution on $[0, T]$.

Since we have
\[
\int_0^{T-h}\|\tau_hv^{\e_k}-v^{\e_k}\|_{L^p(\mathbb{R}^d)}^p dt\le C(p, T)h,
\]
where $C(p,T)$ is independent of $h$. Since $v^{\e_k}\to v$ a.e. in $\mathbb{R}^d\times(0, T)$, by Fatou's lemma
\[
\int_0^{T-h}\|\tau_hv-v\|_{L^p(\mathbb{R}^d)}^p dt\le C(p, T)h.
\]
The time regularity is shown.

Now, we choose $T=n$, $n=1,2, \ldots$. Suppose we have a subsequence $v^{\e_k}$ that tends to a weak solution on $[0, m]$. Then, for $T=m+1$, we pick a subsequence of the current subsequence that  converges as in Proposition \ref{pro:subsequenceconv}. The limit is a weak solution on $[0, m+1)$ and agrees with the weak solution on $[0, m)$. By the standard diagonal argument, we can find a subsequence $v^{\e_k}$ that converges to a function $v\in L_{loc}^{\infty}(0,\infty; L^p(\mathbb{R}^d))\cap L_{loc}^p(0,T; W^{1,p}(\mathbb{R}^d))$, which is a weak solution on any interval $[0, T), T\in (0,\infty)$, and therefore a global weak solution.
\end{proof}

\section*{Acknowledgements}
The work of J.-G. Liu is partially supported by KI-Net NSF RNMS11-07444 and NSF DMS-1514826. We also thank the anonymous referee for insightful suggestions.

\bibliographystyle{unsrt}
\bibliography{NumAnaPde}

\begin{thebibliography}{10}

\bibitem{gangbomccann95}
W.~Gangbo and R.~J. McCann.
\newblock Optimal maps in {M}onge's mass transport problem.
\newblock {\em Comptes Rendus de l'Academie des Sciences-Serie I-Mathematique},
  321(12):1653, 1995.

\bibitem{bb99}
J.-D. Benamou and Y.~Brenier.
\newblock A numerical method for the optimal time-continuous mass transport
  problem and related problems.
\newblock {\em Contemp. Math.}, 226:1--12, 1999.

\bibitem{villani08}
C.~Villani.
\newblock {\em Optimal transport: old and new}, volume 338.
\newblock Springer Science \& Business Media, 2008.

\bibitem{santambrogio15}
F.~Santambrogio.
\newblock Optimal transport for applied mathematicians.
\newblock {\em Prog. Nonlin.}, 87, 2015.

\bibitem{ss2013}
B.~Schmitzer and C.~Schn{\"o}rr.
\newblock Contour manifolds and optimal transport.
\newblock {\em arXiv preprint arXiv:1309.2240}, 2013.

\bibitem{acb17}
M.~Arjovsky, S.~Chintala, and L.~Bottou.
\newblock Wasserstein {GAN}.
\newblock {\em arXiv preprint arXiv:1701.07875}, 2017.

\bibitem{bb00}
J.-D. Benamou and Y.~Brenier.
\newblock A computational fluid mechanics solution to the {M}onge-{K}antorovich
  mass transfer problem.
\newblock {\em Numer. Math.}, 84(3):375--393, 2000.

\bibitem{lps16}
J.-G. Liu, R.~L. Pego, and D.~Slepcev.
\newblock Least action principles for incompressible flows and optimal
  transport between shapes.
\newblock {\em arXiv:1604.03387}, 2016.

\bibitem{arnold66}
V.~Arnold.
\newblock Sur la g{\'e}om{\'e}trie diff{\'e}rentielle des groupes de lie de
  dimension infinie et ses applications {\`a} l'hydrodynamique des fluides
  parfaits.
\newblock In {\em Annales de l'institut Fourier}, volume~16, pages 319--361,
  1966.

\bibitem{stokes1847}
G.~G. Stokes.
\newblock On the theory of oscillatory waves.
\newblock {\em Trans Cambridge Philos Soc}, 8:441--473, 1847.

\bibitem{bhl93}
J.~T. Beale, T.~Y. Hou, and J.~S. Lowengrub.
\newblock Growth rates for the linearized motion of fluid interfaces away from
  equilibrium.
\newblock {\em Commun. Pur. Appl. Math.}, 46(9):1269--1301, 1993.

\bibitem{wu97}
S.~Wu.
\newblock Well-posedness in {S}obolev spaces of the full water wave problem in
  2-{D}.
\newblock {\em Invent. Math.}, 130(1):39--72, 1997.

\bibitem{wu99}
S.~Wu.
\newblock Well-posedness in {S}obolev spaces of the full water wave problem in
  3-{D}.
\newblock {\em J. Am. Math. Soc.}, 12(2):445--495, 1999.

\bibitem{shatahzeng08}
J.~Shatah and C.~Zeng.
\newblock Geometry and a priori estimates for free boundary problems of the
  {E}uler's equation.
\newblock {\em Commun. Pur. Appl. Math.}, 61(5):698--744, 2008.

\bibitem{cmt94}
P.~Constantin, A.~J. Majda, and E.~Tabak.
\newblock Formation of strong fronts in the 2{D} quasigeostrophic thermal
  active scalar.
\newblock {\em Nonlinearity}, 7(6):1495, 1994.

\bibitem{knv07}
A.~Kiselev, F.~Nazarov, and Al. Volberg.
\newblock Global well-posedness for the critical 2{D} dissipative
  quasi-geostrophic equation.
\newblock {\em Invent. Math.}, 167(3):445--453, 2007.

\bibitem{cv10}
L.~A. Caffarelli and A.~Vasseur.
\newblock Drift diffusion equations with fractional diffusion and the
  quasi-geostrophic equation.
\newblock {\em Ann. Math.}, pages 1903--1930, 2010.

\bibitem{loeper06}
G.~Loeper.
\newblock A fully nonlinear version of the incompressible euler equations: the
  semigeostrophic system.
\newblock {\em SIAM J. Math. Anal.}, 38(3):795--823, 2006.

\bibitem{los69}
O.~A. Ladyzhenskaya and R.~A. Silverman.
\newblock {\em The mathematical theory of viscous incompressible flow}.
\newblock Gordon \& Breach New York, 1969.

\bibitem{leibeson83}
L.~S. Leibenson.
\newblock General problem of the movement of a compressible fluid in a porous
  medium.
\newblock {\em Izv. Akad. Nauk Kirg. SSSR}, 9:7--10, 1983.

\bibitem{breit2017}
D.~Breit.
\newblock {\em Existence theory for generalized Newtonian fluids}.
\newblock Academic Press, 2017.

\bibitem{breit2015}
D.~Breit.
\newblock Existence theory for stochastic power law fluids.
\newblock {\em J. Math. Fluid Mech.}, 17(2):295--326, 2015.

\bibitem{jiasverak15}
H.~Jia and V.~Sverak.
\newblock Are the incompressible 3{D} {N}avier--{S}tokes equations locally
  ill-posed in the natural energy space?
\newblock {\em J. Funct. Anal.}, 268(12):3734--3766, 2015.

\bibitem{cl12}
X.~Chen and J.-G. Liu.
\newblock Two nonlinear compactness theorems in ${L}^p(0, {T}; {B})$.
\newblock {\em Appl. Math. Lett.}, 25(12):2252--2257, 2012.

\bibitem{bernis88}
F.~Bernis.
\newblock Existence results for doubly nonlinear higher order parabolic
  equations on unbounded domains.
\newblock {\em Mathematische Annalen}, 279(3):373--394, 1988.

\bibitem{mm12}
A.~Matas and J.~Merker.
\newblock Existence of weak solutions to doubly degenerate diffusion equations.
\newblock {\em Appl. Math.}, 57(1):43--69, 2012.

\bibitem{aac10}
M.~Agueh, A.~Blanchet, and J.~A. Carrillo.
\newblock Large time asymptotics of the doubly nonlinear equation in the
  non-displacement convexity regime.
\newblock {\em J. Evol. Equ.}, 10(1):59--84, 2010.

\bibitem{cl16}
W.~Cong and J.-G. Liu.
\newblock A degenerate p-{L}aplacian {K}eller-{S}egel model.
\newblock {\em Kinet. Relat. Mod.}, 9(4), 2016.

\bibitem{noether71}
E.~Noether.
\newblock Invariant variation problems.
\newblock {\em Transport Theory and Statistical Physics}, 1(3):186--207, 1971.

\bibitem{arnold13}
V.~I. Arnol'd.
\newblock {\em Mathematical methods of classical mechanics}, volume~60.
\newblock Springer Science \& Business Media, 2013.

\bibitem{galdi11}
G.~P. Galdi.
\newblock {\em An introduction to the mathematical theory of the
  {N}avier-{S}tokes equations: steady-state problems}.
\newblock Springer Science \& Business Media, 2011.

\bibitem{beirao09}
H.~B. da~Veiga.
\newblock {N}avier-{S}tokes equations with shear thinning viscosity.
  {R}egularity up to the boundary.
\newblock {\em J. Math. Fluid Mech.}, 11(2):258--273, 2009.

\bibitem{barenblatt52}
G.~I. Barenblatt.
\newblock On self-similar motions of compressible fluids in porous media.
\newblock {\em Prikl. Math.(in Russian)}, 16:679--698, 1952.

\bibitem{damascelli98}
L.~Damascelli.
\newblock Comparison theorems for some quasilinear degenerate elliptic
  operators and applications to symmetry and monotonicity results.
\newblock In {\em Annales de l'Institut Henri Poincare (C) Non Linear
  Analysis}, volume~15, pages 493--516. Elsevier, 1998.

\bibitem{maitre03}
E.~Maitre.
\newblock On a nonlinear compactness lemma in ${L}^p(0,{T}; {B})$.
\newblock {\em Int. J. Math. Math. Sci.}, 2003(27):1725--1730, 2003.

\bibitem{simon86}
J.~Simon.
\newblock Compact sets in the space ${L}^p(0, {T}; {B})$.
\newblock {\em Annali di Matematica pura ed applicata}, 146(1):65--96, 1986.

\bibitem{fks07}
R.~Farwig, H.~Kozono, and H.~Sohr.
\newblock The {H}elmholtz decomposition in arbitrary unbounded domains-a theory
  beyond.
\newblock {\em Proceedings of Equadiff}, pages 77--85, 2007.

\bibitem{mb86}
V.~N. Maslennikova and M.~E. Bogovskii.
\newblock Elliptic boundary value problems in unbounded domains with noncompact
  and nonsmooth boundaries.
\newblock {\em Milan J. Math.}, 56(1):125--138, 1986.

\bibitem{cz52}
A.~P. Calderon and A.~Zygmund.
\newblock On the existence of certain singular integrals.
\newblock {\em Acta Math.}, 88(1):85--139, 1952.

\bibitem{stein16}
E.~M. Stein.
\newblock {\em Singular integrals and differentiability properties of
  functions}, volume~30.
\newblock Princeton university press, 2016.

\bibitem{barbu10}
V.~Barbu.
\newblock {\em Nonlinear differential equations of monotone types in {B}anach
  spaces}.
\newblock Springer Science \& Business Media, 2010.

\end{thebibliography}
\end{document}